\documentclass[12pt]{amsart}
\usepackage[paper=letterpaper, margin = 3cm]{geometry}
\usepackage{graphicx} 
\usepackage{amsfonts}
\usepackage{mathtools}
\usepackage{amsmath}
\usepackage{amssymb}
\usepackage{amsthm}
\usepackage{tikz}
\usetikzlibrary{positioning, arrows}
\usepackage{parskip}

\newtheorem{theorem}{Theorem}[section]
\newtheorem{lemma}[theorem]{Lemma}
\newtheorem{proposition}[theorem]{Proposition}

\theoremstyle{definition}
\newtheorem{definition}[theorem]{Definition}

\newcommand{\m}[1]{\bigl(\begin{smallmatrix}#1\end{smallmatrix}\bigr)}
\newcommand{\M}[1]{\begin{pmatrix}#1\end{pmatrix}}

\newcommand{\mycomment}[1]{}

\newcommand{\ii}{^{-1}}

\title{The Free Functional Calculus in General}
\author{Julian Bushelli }
\date{June 2025}

\begin{document}

\maketitle

\section*{Abstract}

The classical theory of free analysis generalizes the noncommutative (nc) polynomials and rational functions, easily providing such results as an nc analogue of the Jacobian conjecture. However, the classical theory misses out on important functions, such as the Schur complement. This paper presents a generalization of free functions, viewing them as a natural categorial structure: functors between functor categories that commute with natural transformation. We study this construction on general additive categories; we define, characterize and categorize certain sorts of free maps, such as polynomials and rational expressions, and then prove an analogue of the inverse function theorem, demonstrating a natural lifting of proofs into this broader context.

  We then develop some algebraic basis for this theory: we construct vector spaces and an additive category of free polynomials, and define a class of products that allows us to form composition rings on any vector space of free nc polynomials.

\section{Introduction}
The theory of free analysis has been very successful generalization of non-commuting polynomials and rational functions. However, as evidenced in [PPT], there are many functions that clearly have the same flavor as these free functions,  but are only free functions in special cases.

Of primary importance among these functions are the Schur complement, the Sherman-Morrison-Woodbury formula, the principal pivot transform, and the block $2 \times 2$ inverse formula. 

The Sherman-Morrison-Woodbury formula for rank-$k$ updates tells us that \[(A+UCV)^{-1}= A\ii-A\ii U(C\ii+VA\ii U)\ii VA\ii\]
Where $A$ is $n\times n$, $C$ is $k\times k$ (both invertible) and $U$ is $n\times k$ and $V$ is $k\times n$. When $k = n$ this is simply an nc rational function, and hence a free function; typically, however, $k$ is much smaller than $n$ (often $k=1$). 

The principal pivot transform is the involution map 
\[ppt_A\M{A & B\\C&D} = \M{A\ii& -A\ii B\\ -CA\ii & D-CA\ii B}.\]

This is defined on any block $2\times 2$ matrix with invertible $A$ (or $D$, with adjustment). In the case where all the blocks are square of the same size, the free inverse function theorem easily tells us that this function is invertible; further, the authors of [PPT] were able to extract monotonicity certificates to verify a positivity conjecture based off the square results similar to [RSoS]. 

Finally, the block $2\times 2$ inversion formula tells us that as long as $A$ (or $D$ with obvious adjustment) is invertible, then \[\M{A & B\\C&D}\ii = \M{A\ii+A\ii B(D-CA\ii B)\ii CA\ii&-A\ii B(D-CA\ii B)\ii\\-(D-CA\ii B)\ii CA\ii&(D-CA\ii B)\ii}.\]

As with all of these, the structure clearly suggests a free function theoretic approach. But we have a problem; the current theory requires $A, B, C, D$ to be square matrices of the same size, or otherwise operators from a space to itself.

It has been shown that invariant structure preserving functions turn out to be free functions [Inv]; as a result, it is natural to expand the notion of a free function, discussing a yet broader sort of structure preservation. We develop these nc free maps as a natural categorial structure: free maps are those functors between functor categories that commute with natural transformations, thus preserving the structure of the functor categories.

We develop a theory of free maps over additive categories that allows the techniques of free analysis to be applied to functions with not-necessarily-composable entries. Some categorization of these free maps is made, and an identification which justifies viewing free maps in terms of abstract variables. The notion of a derivative is constructed, and an inverse function theorem is demonstrated. 

In the process we observe a guiding principle: \textit{natural proofs lift naturally.} Hence in addition to an inverse function theorem, we lift an automatic differentiation theorem and a bianalyticity theorem from free analysis. 

Finally we restrict ourselves to the setting most analogous to that of free polynomials, and construct vector space, additive category, and ring structures, thus clarifying the underlying algebraic structure.

\section{Preliminary Review}
We will now provide a brief overview of operator theory and category theory. The following section includes an overview of the standard operator theory used within this paper, as well as a brief survey of some important results in free analysis. Following that is an introduction to the language of category theory; of primary importance is the notion of a functor (in many contexts called a representation) and the maps between functors. 

\subsection{Operator Theory}
We will give only a brief introduction to general operator theory, focusing primarily on matrices, before discussing the more specific field of free analysis. Those who are familiar with the broader topics of operator theory will see how deeply the theorems in this paper can extend, while those who stick to the matrix case will still find the whole paper comprehensible and informative. 

In operator theory we are often interested in linear transformations between vector spaces; indeed, the words operator, transformation, and map are all more or less synonymous with function, their difference only being a ``flavor" conventionally ascribed. 

This so-called ``linearity" is an abstraction of the integral, specifically the property \[\int \alpha f(x)+ g(x)dx = \alpha\int f(x)dx + \int g(x)dx.\]

Thus we say that an operator between vector spaces $A:V\rightarrow W$ is \textbf{linear} if \[A(\alpha x+y) = \alpha A(x)+A(y)\] for $x, y\in V$ and some scalar $\alpha$. 

The integral is an important example of a linear operator, but is not alone; the derivative is linear, as are many other more esoteric operations. 

Due to its relative simplicity, there has been a great study of the operator theory of finite-dimensional vector spaces; indeed, all linear operators between finite dimensional vector spaces are continuous, bounded, and described by matrices. 

For example, consider the operator $A: \mathbb R_2[x] \rightarrow \mathbb R$ given by $Af = \int_{-1}^1f(x)dx$. 

Note that $\mathbb R_2[x]$ is the set of polynomials with real coefficients that have degree at most two; that is, $f \in \mathbb R_2[x]$ will look like $f(x) = ax^2+bx+c$ where $a,b,c \in \mathbb R$; this can be represented by a vector $\m{a\\b\\c}\in \mathbb R^3$. 

Then we have \[\int_{-1}^1 ax^2+bx+c dx = 2a+\frac{2}{3}c = \M{2 & 0 &\frac{2}{3}}\M{a \\b\\c}.\]
Thus $A$ can be considered to be the matrix $\m{2&0&\frac{2}{3}} \in M_{1,3}(\mathbb R): \mathbb R^3 \rightarrow \mathbb R$. 

Similarly, the derivative operator $\frac{d}{dx}: \mathbb R_2[x]\rightarrow \mathbb R_1[x]$ can be represented as a matrix $A: \mathbb R^3 \rightarrow \mathbb R^2$. If we identify $ax^2+bx+c = \m{a\\b\\c}$, then the derivative is given by \[\M{2&0&0\\0&1&0}\M{a \\b\\c}= \M{2a\\b}.\]
In general we concern ourselves with $A:\mathbb C^n \rightarrow \mathbb C^m$, which is to say $A \in M_{m,n}(\mathbb C)$, and then apply the sort of identities used above when we need to. 

We will briefly discuss the so-called hermitian or self-adjoint matrices. A hermitian matrix satisfies the relation $A^* = A$; that is, the conjugate transpose of $A$ is itself $A$. Hermitian matrices are important in the theory of matrix positivity,  as most positivity theorems are proven on self-adjoint matrices. We say that a hermitian matrix is positive if all of its eigenvalues are positive; this definition is possible because a hermitian matrix has strictly real eigenvalues.

\subsubsection{Univariate Matrix Functions}
Evaluation of univariate functions on square matrices has a long and storied history; they in some sense form the backbone of matrix theory.

Suppose we are given a single variable polynomial $p \in \mathbb C[x]$ with complex coefficients, namely \[p(x) = \sum_{i=0}^nk_ix^i.\] Then if we select any bounded operator $X$ from a Banach space to itself, we can evaluate $p$ by substituting $X$ for $x$, where we define $X^0$ to be the identity operator. When dealing with matrices, this leads to such results as the Cayley-Hamilton theorem. 

\begin{theorem}[Cayley-Hamilton]
    A matrix satisfies its own characteristic polynomial.
\end{theorem}

The characteristic polynomial of $X$ is given by $p(\lambda) = det(X-\lambda I)$, and has as its zeros the eigenvalues of $X.$ The Cayley-Hamilton theorem then tells us that $p(X) =0$ (the appropriate zero matrix); this result is subtler than it first appears.

A famous result in matrix analysis is that almost every matrix is similar to a diagonal matrix of its eigenvalues That is, $X = S\ii DS$ for some invertible matrix $S$, and $D$ with the eigenvalues of $X$ on the diagonal and zero everywhere else. Then if we consider any polynomial of degree $n$, we see that \[p(X) = \sum_{i=0}^nk_iX^i= \sum_{i=0}^nk_i(S\ii DS)^i=\sum k_iS\ii D^iS=S\ii p(D)S\]
because $S\ii S=I$. But in turn, it is easy to see that with a diagonal matrix \[p\M{a_1 & \dots & 0\\\vdots&\ddots &\vdots\\0&\dots &a_n} = \M{p(a_1) & \dots & 0\\\vdots&\ddots &\vdots\\0&\dots &p(a_n)}.\]
A very similar method can be employed with non-diagonalizable matrices using either the Schur triangularization or the Jordan canonical form; the difference there is that above the diagonal there will be some non-zero entries. 

Consider now the standard nilpotent matrix of size $n$, which has ones on its first super-diagonal and zeros elsewhere, such as \[N_5 = \M{0&1&0 & 0 & 0
\\0&0&1&0&0\\0&0&0&1&0\\0&0&0&0&1\\0&0&0&0&0}.\]
The standard nilpotent matrix of size $n$ satisfies the equation $N_n^n = 0$. Indeed, each power of $N$ has the ones on the next super-diagonal; that is, \[N_5^2 = \M{0&0&1&0&0\\&&0&1&0\\&&&0&1\\&&&&0\\&&&&0}\] where there is a zero anywhere left blank. As usual, $N_n^0 = I_n$, the identity matrix of size $n$. Thus plugging $N_5$ into a univariate polynomial, we can read off from the resulting matrix the coefficients of $x^0, \dots, x^4$ simply by reading the top line of the resulting matrix. As an example, if we have the polynomial $p(x) = 1+4x+3x^3$ and plug in the nilpotent $N_3$, we obtain \[p(N_3) = I+4N_3+3N_3^3 = \M{1 &4&0\\&1&4\\&&1}.\] We thus can read the coefficients; note that can retain the powers of $x^3$ simply by choosing a larger matrix. This is of little use when dealing with such small polynomials, but can be revealing when dealing with more elaborate functions. 

\subsubsection{Classical Polynomial Generalizations}

A historical generalization of noncommutative polynomials is the so-called holomorphic functional calculus, which takes certain complex-analytic functions and evaluates them on operators. 

Given a complex holomorphic function $f(z) = \sum_{i=0}^\infty k_iz^i$ that has some radius of convergence $r$, we can define a function on matrices by  \[f(T)= \sum_{i=0}^\infty k_i T^i\]
so long as $\|T\| < r$. Thus if $f$ is entire, like the exponential function, $f(T)$ is defined for any matrix. Thus $e^T = \sum \frac{T^i}{i!}$. 

Note that for analytic functions evaluation on the nilpotents $N_n$ can be valuable for numerical approximation; we can find an arbitrarily good polynomial approximation by evaluating it on a matrix input. 

Another historical generalization of the univariate polynomials on operators is the Borel functional calculus; if we choose a self-adjoint operator on a Hilbert space, or in the matrix case a hermitian matrix, we can evaluate it on any Borel-measurable function and get reasonable outputs. This framework enables good positivity results.

\subsection{Free Analysis}
Free analysis is concerned with matrix-valued functions that behave uniformly across all matrix sizes, by respecting direct sums and similarity and thus generalizing the free noncommutative polynomials. 

Free functions are defined on nc sets of matrices, which are defined as $d$-tuples of matrices of all sizes that respect direct sums and similarity; that is, they are subsets of  \[M^d(k) := \bigsqcup_{n\in N} M^d_n( k)\] where $k$ is often taken to be $\mathbb C$. When a field has been fixed, we write simply $M^d$. Thus a full nc set is a subset $U \subset M^d$ where if $X, Y = (X_1, \dots, X_d), (Y_1, \dots, Y_d) \in U$ implies \[X\oplus Y = (X_1 \oplus Y_1, \dots, X_d\oplus Y_d) = (\m{X_1 &\\&Y_1}, \dots, \m{X_d &\\&Y_d})\in U\] (where $\m{X &\\&Y}$ is a block diagonal matrix with 0s on the off diagonals) as well as \[S\ii XS = (S\ii X_1S, \dots, S\ii X_dS) \in U\] for every invertible $S$ of the appropriate size. We say that an nc set is open if it is open in the standard norm topology at every matrix size $n$. 

We now define a free function $f: U\subset M^d \rightarrow M^{\tilde d}$ as a function which respects intertwining, which is to say a function satisfying \[X\Gamma = \Gamma Y \implies f(X)\Gamma = \Gamma f(Y)\]
where $\Gamma$ is an appropriately sized not necessarily square matrix. These are entrywise equations as above, interpreted as \[\begin{split}
    (X_1\Gamma, \dots, X_d\Gamma) &= (\Gamma Y_1, \dots, \Gamma Y_d) \\ &\Downarrow  \\(f_1(X)\Gamma, \dots, f_{\tilde{d}}(X)\Gamma) &= (\Gamma f_1(Y), \dots, \Gamma f_{\tilde{d}}(Y)).
\end{split}\]
This relationship is is encoded in the commuting diagram 

\begin{center} \begin{tikzpicture}[node distance=2cm, auto]
  \node (X) {$X$};
  \node (fx) [right of=X] {$f(X)$};
   \node (Y) [below of=X] {$Y$};
  \node (fy) [right of=Y] {$f(Y)$};
  \draw[->] (X) to node [swap] {$\Gamma$} (Y);
  \draw[->] (X) to node {$f$} (fx);
  \draw[->] (Y) to node {$f$} (fy);
  \draw[->] (fx) to node [swap] {$\Gamma$} (fy);
\end{tikzpicture}\end{center}

Key to our generalization, this is the diagram that encodes a natural transformation in category theory. 

A foundational proposition of free analysis is the following. \begin{proposition}
    A function $f: U\in M^d\rightarrow M^{\tilde d}$ is free if and only if it respects direct sums and similarity. 
\end{proposition}

That is, $f(X\oplus Y) = f(X) \oplus f(Y)$ and $f(S\ii XS) = S\ii f(X)S$ whenever $S$ is invertible.

Using this characterization of free functions, we can easily prove a much more profound result; every continuous free function is automatically analytic [KVV]. \begin{proposition}
    For an open set $U\subset M^d$, if a free function $f: U \rightarrow M^{\tilde d}$ is continuous then it is analytic and for small enough $H$\[f\M{X & H\\&X} = \M{f(X) & Df(X)[H] \\ &f(X)}\]
\end{proposition}
where $Df(X)[H]$ is the standard Fr\'echet derivative evaluated in direction $H$. 

Among many other results, the following theorem has been shown [IFT].

\begin{theorem}[Free Inverse Function Theorem]
    Let $U \in M^d$ be an open set. For a free function $f: U \rightarrow M^d$, the following are equivalent.
    \begin{itemize}
        \item $Df(X)$ is non-singular for all $X\in U$
        \item $f$ is injective. 
        \item $f^{-1}: f(U) \rightarrow U$ exists and is a free function.  
    \end{itemize}
\end{theorem}

Thanks to [Groth] and the global nature of the above result, we gain the following theorem. 
\begin{theorem}[Free Jacobian Theorem]
    Let $f$ be a free nc polynomial evaluated on matrix inputs. The following are equivalent.
    \begin{itemize}
        \item $Df(X)$ is non-singular for all $X\in U$
        \item $f$ is injective.
        \item $f$ is bijective.
        \item $f\ii$ exists and is also a free nc polynomial. 
    \end{itemize}
\end{theorem}

There are many other notable results, such as in [Prop] and [Loci]; among the most important is the following analogue of Hilbert's 17th problem, due to Helton [SoS]. 
\begin{theorem}
    A noncommutative polynomial is matrix positive if and only if it is a sum of squares. 
\end{theorem}
We say that $p$ is \textbf{matrix positive} if $p(X)$ is a positive semidefinite matrix (that is, hermitian with positive eigenvalues) for every tuple of hermitian matrices $X$. A polynomial in matrices is said to be a sum of squares, or sum of hermitian squares, if $p(X, X^*) = \sum_ik_ip_i(X, X^*)^*p_i(X,X^*)$ where it takes as argument a tuple $X = (X_1, \dots, X_n)$ and the tuple $X^* = (X_1^*, \dots, X_n^*)$ where $X_i^*$ is the conjugate transpose of $X_i$. 

Many refinements to this theorem have been made, such as [SoHS], [H17]. 

\subsection{Category Theory}
We include here a brief introduction covering the basic concepts required in this discussion.

\begin{definition}[Category]
    A \textbf{category} $C$ consists of the following 
    \begin{itemize}
    \item A collection of \textbf{objects}, denoted $ob(C)$
    \item For every $u, v\in ob(C)$, a collection of \textbf{morphisms} from $u$ to $v$, denoted $mor(u, v)$. 
    \begin{itemize}
        \item If $x \in mor(u, v)$ and $y\in mor(v, w)$ then there is a morphism $yx := y\circ x \in mor(u, w)$; this composition must be associative. 
        \item There exists an element $1 \in mor(u) := mor(u,u)$ called the \textbf{identity} so that $1x = x$ and $x1 =x$ for any $x$ with appropriate domain and codomain. 
    \end{itemize}

    \end{itemize} 
    
\end{definition}

Morphisms can be usefully thought of as a generalization of functions; indeed, in the category of sets the morphisms are exactly the functions, and in the category of vector spaces they are exactly the linear transformations. As such, the language of category theory can be usefully leveraged to make function-theoretic arguments. 

A category is called \textbf{locally small} if $mor(v,w)$ is a set for any objects $v,w$. It is called \textbf{small} if its collection of objects also forms a set. 

For the majority of this paper, we will be assuming that the categories are at least locally small, for the purposes of using such notation as $\in$ and $\subset$, though the proofs are not restrictive. 

Examples of categories include: 
\begin{itemize}
    \item Any group, with one vacuous object and its elements as morphisms;
    \item The category of vector spaces over a fixed field $k$, with vector spaces as objects and linear transformations as morphisms;
    \item The category of $R$-modules, with $R-$module homomorphisms as its morphisms; 
    \item the subsets of some set $A$, with a single morphism $a \rightarrow b$ if $a \subset b$;
    \item any partial ordering, with a morphism $a\rightarrow b$ if $a\leq b$;
    \item the category of topological spaces, with continuous functions as morphisms.
\end{itemize}

We define below a type of category that will be crucial throughout this paper. 

\begin{definition}[Free Category]
    A \textbf{free category} (or path category) $Q$ is a category induced by a quiver (i.e. multidigraph) $Q'$ where $ob(Q)$ consists of the vertices of $Q'$ and $mor(u, v)$ consists of all paths in $Q'$ from $u$ to $v$. The identity in $mor(u)$ is taken to be the empty path from $u$ to itself.
    
    In this paper we also refer to $Arcs(Q)$, which consists of those morphisms of $Q$ that are paths of length 1. Which is to say, $Arcs(Q)$ consists of the arcs of $Q'$.      
\end{definition}
    Note that the convention for writing morphisms in a free category follows that of function composition, and is hence opposite that for writing paths in a multidigraph. 
    Thus for the free category generated by the graph below the morphism from $u$ to $w$ is called $yx$.
    \begin{center}
    \begin{tikzpicture}[node distance=2cm, auto]
        \node (u) {$u$};
        \node (v) [right of = u] {$v$};
        \node (w) [right of = v] {$w$};
        \draw[->] (u) to node {$x$}(v);
        \draw[->] (v) to node {$y$} (w);
    \end{tikzpicture}
\end{center} 

We will soon define the notion of an additive category, but first we must deal with a few important notions. 
\subsubsection{Functoriality}
A homomorphism of categories is called a functor, as defined below. 

\begin{definition}
    Let $C, D$ be categories. A \textbf{functor} $f: C\rightarrow D$ is a mapping between categories satisfying the following. 
    
    For all $u, v, w \in ob(C)$, and for all $x\in mor(u, v)$ and $y\in mor(v, w)$:  
    \begin{itemize}
        \item There is an object $f(u) \in ob(D)$.
        \item There is a morphism $f(x) \in mor(f(u), f(v))$ for each $x \in mor(u, v)$.
        \item For the identity $1_{u} \in mor(u)$, the image $f(1) = 1_{f(u)} \in mor(f(u))$ is the identity. 
        \item The functor respects composition, that is $f(yx) = f(y)f(x)$.
    \end{itemize}
\end{definition}

This definition works for any two categories, though sometimes the only valid functor maps all objects to a single object and all morphisms to that object's identity. 

Commonly we will deal with functors $F: Q \rightarrow C$ from a free category to another category. In this case each arc $x \in Arcs(Q)$ is mapped to any morphism of $C$ with the correct domain and codomain, and the rest of the morphisms are defined by products $F(y)F(x)$. 

For example, the diagram below describes a functor from a free category to that of finite dimensional vector spaces. \begin{center}
    \begin{tikzpicture}[node distance=2cm, auto]
        \node (u) {$u$};
        \node (v) [right of = u] {$v$};
        \node (w) [right of = v] {$w$};
        \draw[->] (u) to node {$x$}(v);
        \draw[->] (v) to node {$y$} (w);
    \end{tikzpicture}
\end{center} 
\[\downarrow{F}\]
\begin{center}
    \begin{tikzpicture}[node distance=2cm, auto]
        \node (u) {$\mathbb C^1$};
        \node (v) [right of = u] {$\mathbb C^2$};
        \node (w) [right of = v] {$\mathbb C^3$};
        \draw[->] (u) to node {$\m{\alpha \\ \beta}$}(v);
        \draw[->] (v) to node {$\m{a & b \\ c&d \\ g & h}$} (w);
    \end{tikzpicture}
\end{center} 

In this case, for arbitrary choice of such matrices,  \[F(yx) = \M{a & b \\ c&d \\ g & h}\M{\alpha \\ \beta} = \M{a\alpha + b\beta \\ c\alpha + d\beta \\ g\alpha +h\beta}.\]

Now that the idea of a functor has been presented, we will also define maps between functors.
 \begin{definition}
     A \textbf{natural transformation} $\eta$ between functors $F,G : C\rightarrow D$ consists of a morphism $ \eta_u \in mor_D(F(u), G(u))$ for each object $u\in ob(C)$ so that the following diagram commutes for every $u, v\in ob(C)$ and $x\in mor(u, v)$. 
     \begin{center} \begin{tikzpicture}[node distance=2cm, auto]
  \node (X) {$F(u)$};
  \node (fx) [right of=X] {$F(v)$};
   \node (Y) [below of=X] {$G(u)$};
  \node (fy) [right of=Y] {$G(v)$};
  \draw[->] (X) to node [swap] {$\eta_u$} (Y);
  \draw[->] (X) to node {$F(x)$} (fx);
  \draw[->] (Y) to node {$G(x)$} (fy);
  \draw[->] (fx) to node [swap] {$\eta_v$} (fy);
\end{tikzpicture}\end{center}

Which is just to say that the equation \[G(x) \eta_u = \eta_v F(x)\]
always holds.  \end{definition}

The following is a simple example of transformations between functors of the aforementioned path diagram. 
\begin{center}
    \begin{tikzpicture}[node distance=2cm, auto]
        \node (u) {$\mathbb C^1$};
        \node (v) [right of = u] {$\mathbb C^2$};
        \node (w) [right of = v] {$\mathbb C^3$};
        \node(u') [below of = u] {$\mathbb C^2$};
        \node(v') [below of = v] {$\mathbb C^1$};
        \node(w') [below of = w] {$\mathbb C^2$};

        \draw[->] (u) to node {$\m{1 & 1}$} (u');
        \draw[->] (v) to node {$\m{2 \\1}$} (v');
        \draw[->] (w) to node {$\m{1 & 0\\0&1\\0&0}$} (w');

        \draw[->] (u') to node {$\m{2\alpha \\ \beta}$}(v');
        \draw[->] (v') to node {$\m{a& b}$} (w');
                
        \draw[->] (u) to node {$\m{\alpha & \beta}$}(v);
        \draw[->] (v) to node {$\m{2a & 2b & c \\ a & b & d}$} (w);
    \end{tikzpicture}
\end{center} 

Every functor has an identity transformation to itself, defined by choosing $\eta_x$ to be $id_{Fx}$. Similarly, if we have \begin{center}
    \begin{tikzpicture}[node distance=2cm, auto]
        \node (u) {$F$};
        \node (v) [right of = u] {$G$};
        \node (w) [right of = v] {$H$};
        \draw[->] (u) to node {$\eta$}(v);
        \draw[->] (v) to node {$\xi$} (w);
    \end{tikzpicture}
\end{center}  
where $F, G, H$ are functors and $\eta, \xi$ are natural transformations, then the composition of the two defines a natural transformation $\xi\eta: F\rightarrow H$. 

Another example are group representations. A representation of a group is functor, and a similarity transformation is a natural transformation. 

Thus natural transformations have all of the properties of a morphism in a category, and so we can define the following. 
\begin{definition}
    The \textbf{functor category} $C^D$ has as its objects all of the functors $F: D\rightarrow C$, and as its morphisms the natural transformations between such functors. 
\end{definition}

Recall that in some sense a functor $F \in C^D$ is a collection of objects and morphisms in $C$ that has the structure of the category $D$; a natural transformation $\eta: F\rightarrow G$ is some morphisms from the objects of $F$ to the objects of $G$ that makes corresponding morphisms in $F$ and $G$ act similarly to each other. 

\subsubsection{Additive Categories}
We also need to appeal to properties of additive categories; these are the categories that act like vector spaces, in some sense, and they play that role in this theory. 

\begin{definition}
    An \textbf{additive category} $C$ is a category that satisfies the following conditions.
    \begin{itemize}
        \item $mor(u, v)$ is an abelian group. It is often called $Hom(u,v)$ in this context. 
        \item Composition of morphisms is bilinear; that is $(x+y)\circ z = (x\circ z)+(y\circ z)$, etc. 
        \item All finite products (equivalently, coproducts) exist. 
        \item It contains a zero object. 
    \end{itemize}
\end{definition}

Several of these details are quite technical; for our purposes, an additive category is effectively indistinguishable from a category of $R$-modules. Thus, we will stick to such examples for the time being. For further information about the categorial details, see [C].

Examples of additive categories include:
\begin{itemize}
    \item The category of abelian groups, the homomorphisms as its morphisms;
    \item the category of Banach spaces with linear transformations;
    \item The category of finite dimensional vector spaces, with linear transformations as morphisms;
    \item the category of Hilbert spaces with linear transformations;
    \item The category of vector spaces over $k$ with linear transformations;
    \item the category of $R$-modules with homomorphisms as morphisms.
\end{itemize}

Clearly, these spaces have some very nice properties. Two essential but non-obvious properties are the following:

\begin{itemize}
    \item The direct sum $\oplus$ is both the product and the coproduct.
    \item a morphism $f: u \oplus v\rightarrow s\oplus t$ can be represented as a matrix of morphisms $f = \m{f_{su} & f_{sv}\\f_{tu}&f_{tv}}$ where $f_{ij}: j\rightarrow i$.
\end{itemize}

There is a further idea of an \textbf{abelian category}, which is somewhat more commonly used. However, it is not necessary in this discussion, and excludes certain categories (such as the category of Banach spaces) that we do not wish to lose. 

\subsection{Three running examples}
Throughout this paper we will frequently use the following categories as indexing categories. 

\subsubsection{Example 1: Classical Free Analysis}
A primary example will be the classical free analysis, given as a special case of the more general free calculus presented here. This is more fully explained in the first example at the end of first section. 

The following category --- and those with $n$ vertices, rather than 2 --- provides the free analysis case. 

\begin{center}
    \begin{tikzpicture}[node distance=2cm, auto]
        \node (u) {$u$};
        
        \path[->] (u) edge [loop left] node[auto] {$x$} ();
        \path[->] (u) edge [loop right] node[auto] {$y$} ();
    \end{tikzpicture}
\end{center} 

That is, there is a single vertex, and two loops. Thus the path set is $\langle x, y\rangle$, that is the set of all words in the letters $x, y$; these then (read from right to left, as is standard for functions) are the morphisms, with the empty word as the identity morphism on $u$. 

\subsubsection{Example 2: the Free Category $\mathcal S$ch}
We will consider the free category generated by the following quiver. We will call this category $\mathcal Sch$, because this is the appropriate domain for the schur complement. 

\begin{center}
    \begin{tikzpicture}[node distance=2cm, auto]
        \node (u) {$u$};
        \node (v) [right of = u] {$v$};

        \draw [->] (u) to [bend left] node {$x_{21}$} (v);
        \draw[->] (v) to [bend left] node {$x_{12}$} (u);
        \path[->] (u) edge [loop left] node[auto] {$x_1$} ();
        \path[->] (v) edge [loop right] node[auto] {$x_2$} ();
    \end{tikzpicture}
\end{center}

The set $Arcs(Sch) = \{x_{1}, x_2, x_{12}, x_{21}\}$. The set of morphisms includes such morphisms as $x_{21}x_1$, $x_{12}x_2x_{21}$, $x_2^2x_{21}$, $(x_{12}x_{21})^nx_1^m$, but certainly not $x_1y_2$ or $x_{12}^2$. 

\subsubsection{Example 3: The symmetric group $S_3$}
The third example will be the symmetric group $S_3$, considered as a category. As a category, it has a single vacuous object and 6 morphisms, one for each group element. Composition is defined as the group operation; thus each morphism is an isomorphism. 

In cycle notation, the six morphisms are 
\begin{itemize}
    \item $id = (1)$
    \item $(12)$
    \item $(13)$
    \item $(23)$
    \item $(123) = (23)(12) =(13)(23)=(12)(13)$
    \item $(132) = (12)(23) = (23)(13) = (13)(12)$
\end{itemize}

\section{Free Maps}
\subsection{nc Categories}
Let $Q$ be a category. 

Now choose an additive category $C$, and consider the functor category $C^{Q}$, the category of (covariant) functors from $Q$ to $C$ with morphisms being natural transformations (when $C$ is a vector space category, this is often called a representation).

\begin{definition}
    Firstly, we define a \textbf{natural automorphism} as a natural transformation wherein all of the morphisms are automorphisms.

    Let $C$ be an additive category. 
    
    We say an induced subcategory $U$ of the functor category $C^Q$ is a \textbf{full nc-subcategory} of $C^Q$ so long as it is closed under direct sum and natural automorphism. 

    That is: if $X, Y \in U$, then $X\oplus Y \in U$; if $S$ is a natural automorphism then $S^{-1}XS \in U$.  

   We say that multiple categories, say $\{Q_i\}$, have the same objects if there is some identification of objects $\mathcal V$, consisting of a set $\mathcal V_o = \{u_1, \dots, u_m\}$ and some bijections $\mathcal V_i: \mathcal V_o \rightarrow ob(Q_i)$ for each category $Q_i.$  
    This identity will generally be implicit.
\end{definition}

\subsection{Free Maps}
We will use the terms ``natural transformation'' and ``morphism of $C^Q$" interchangeably, as is appropriate for functor categories.

\begin{definition}
    Let $U\subset C^Q$ and $V \subset C^R$ be full nc subcategories. Then a free map $f: U\rightarrow V$ is a functor that respects natural transformations.
    That is: for any $X,Y \in U$ and natural transformation $\Gamma$ \[ \Gamma X=Y\Gamma \implies \Gamma f(X)= f(Y) \Gamma ,\]
which is to say that the following square commutes. 
\begin{center} \begin{tikzpicture}[node distance=2cm, auto]
  \node (X) {$X$};
  \node (fx) [right of=X] {$f(X)$};
   \node (Y) [below of=X] {$Y$};
  \node (fy) [right of=Y] {$f(Y)$};
  \draw[->] (X) to node [swap] {$\Gamma$} (Y);
  \draw[->] (X) to node {$f$} (fx);
  \draw[->] (Y) to node {$f$} (fy);
  \draw[->] (fx) to node [swap] {$\Gamma$} (fy);
\end{tikzpicture}\end{center}
\end{definition}
Note that this definition does not essentially depend on the additivity of $C;$ any functor categories $C^Q$ and $C^R$ admit functors that commute in this way, restricted though such functors might be. However, we restrict ourselves to the case where $C$ is additive in this discussion.

One important consequence of the definition is that the objects are not changed by free maps. Indeed, if we consider a mapping $\Gamma_v X(v)$ from the object $v$, then if $f(X)(v) \not\cong X(v)$, then $\Gamma_v f(X)(v)$ is not defined. 

It also follows that a composition of free maps is itself a free map. 

\subsubsection{An Introductory Example}

Let us choose $C$ to be the category of finite dimensional Hilbert spaces over $\mathbb C$. We will choose $Q$ to be $\mathcal Sch$. Then the following is an element of the functor category $C^Q$. 

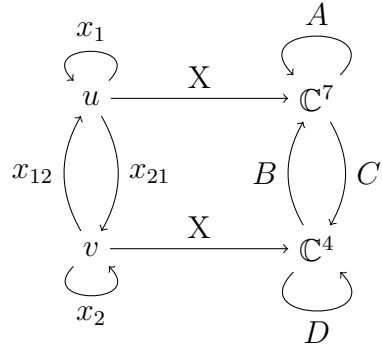
\begin{center}
    \begin{tikzpicture}[node distance=2cm, auto]
        \node (u) {$u$};
        \node (v) [below of = u] {$v$};

        \draw [->] (u) to [bend left] node {$x_{21}$} (v);
        \draw[->] (v) to [bend left] node {$x_{12}$} (u);
        \path[->] (u) edge [loop, out = 45, in = 135, looseness = 4, swap] node[auto] {$x_1$} (u);
        \path[->] (v) edge [loop, out = 225, in = 315, looseness = 4, swap] node[auto] {$x_2$} (v);

        \node (w) [node distance = 3cm, right of = u] {$\mathbb C^7$};
        \node (t) [below of = w] {$\mathbb C^4$};

        \draw [->] (w) to [bend left] node {$C$} (t);
        \draw[->] (t) to [bend left] node {$B$} (w);
        \path[->] (w) edge [loop, out = 45, in = 135, looseness = 4, swap] node[auto] {$A$} (w);
        \path[->] (t) edge [loop, out = 225, in = 315, looseness = 4, swap] node[auto] {$D$} (t);

        \draw[->] (u) to node {X} (w);
        \draw[->] (v) to node {X} (t);
        
    \end{tikzpicture}
\end{center}

Here $C: \mathbb C^4\rightarrow \mathbb C^7$ is a linear transformation; the rest are also linear. We say $X(u) = \mathbb C^7, X(v) = \mathbb C^4$, as well as $X(x_{21}) = C$, $X(x_1) = A$, and so forth. 

Recall that this is in some sense an incomplete picture of $Q$; the remainder of the morphisms are to be filled out by composing the existing morphisms appropriately.

Since this is a functor and thus preserves the composition of morphisms, $x_{21}x_1^2x_{12} \in mor(v)$ implies $X(x_{21})X(x_1)^2X(x_{12}) = CA^2B \in mor(\mathbb C^4) \subset B(\mathbb C^4)$; thus the fact that this multiplication of matrices makes sense is encoded within the free category $Q$. 

By $\Gamma Y = X\Gamma$ we mean that there is a natural transformation $\Gamma$ between $X, Y \in C^Q$ consisting of some linear transformations $\Gamma_u, \Gamma_v$ making all of the squares involving $\Gamma$ commute. Note that $Y(u) = \mathbb C^2$ and $X(u) = \mathbb C^7$, so $\Gamma_u: \mathbb C^2 \rightarrow \mathbb C^7$, and $\Gamma_v: \mathbb C^{13}\rightarrow \mathbb C^4$

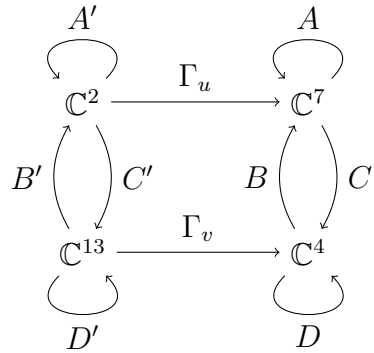
\begin{center}
    \begin{tikzpicture}[node distance=2cm, auto]
        \node (u) {$\mathbb C^2$};
        \node (v) [below of = u] {$\mathbb C^{13}$};

        \draw [->] (u) to [bend left] node {$C'$} (v);
        \draw[->] (v) to [bend left] node {$B'$} (u);
        \path[->] (u) edge [loop, out = 45, in = 135, looseness = 4, swap] node[auto] {$A'$} (u);
        \path[->] (v) edge [loop, out = 225, in = 315, looseness = 4, swap] node[auto] {$D'$} (v);

        \node (w) [node distance = 3cm, right of = u] {$\mathbb C^7$};
        \node (t) [below of = w] {$\mathbb C^4$};

        \draw [->] (w) to [bend left] node {$C$} (t);
        \draw[->] (t) to [bend left] node {$B$} (w);
        \path[->] (w) edge [loop, out = 45, in = 135, looseness = 4, swap] node[auto] {$A$} (w);
        \path[->] (t) edge [loop, out = 225, in = 315, looseness = 4, swap] node[auto] {$D$} (t);

        \draw[->] (u) to node {$\Gamma_u$} (w);
        \draw[->] (v) to node {$\Gamma_v$} (t);
        
    \end{tikzpicture}
\end{center}

This means that each of the squares of the following sort commute, where $t, w\in obQ$ and $a \in mor(t,w)$. 

\begin{center} \begin{tikzpicture}[node distance=2cm, auto]
  \node (X) {$Y(t)$};
  \node (fx) [right of=X] {$X(t)$};
   \node (Y) [below of=X] {$Y(w)$};
  \node (fy) [right of=Y] {$X(w)$};
  \draw[->] (X) to node [swap] {$Y(a)$} (Y);
  \draw[->] (X) to node {$\Gamma_t$} (fx);
  \draw[->] (Y) to node {$\Gamma_w$} (fy);
  \draw[->] (fx) to node [swap] {$X(a)$} (fy);
\end{tikzpicture}\end{center}

This diagram is interpreted as $\Gamma_wY(a) = X(a)\Gamma_t$. In particular, this implies the following four equations: \[ \begin{matrix}
\Gamma_uA' = A\Gamma_u &\in B(\mathbb C^2, \mathbb C^7)\\
    \Gamma_uB' = B\Gamma_v &\in B(\mathbb C^{13}, \mathbb C^7)\\ 
    \Gamma_v C' = C\Gamma_u &\in B(\mathbb C^2, \mathbb C^4)\\
    \Gamma_v D' = D\Gamma_v&\in B(\mathbb C^{13}, \mathbb C^4)
\end{matrix}.\]
Now that we have explored a natural transformation, let us consider a free map $f: C^Q \rightarrow C^R$, where $R$ is the free category 

\begin{center}
    \begin{tikzpicture}[node distance=1cm, auto]
        \node (u) {$u$};
        \node (v) [below of = u] {$v$};

        \draw [->] (u) to node {$x_{21}$} (v);
        \path[->] (u) edge  [loop, out = 45, in = 135, looseness = 4, swap] node [auto] {$x_1$} (u);

    \end{tikzpicture}
\end{center}

all of whose morphisms are $x_1^nx_{21}$ for some natural number $n$. Then consider the map $f: C^Q\rightarrow C^R$, and in particular we have $X \mapsto_f f(X)$. 
\begin{center}
 \begin{tikzpicture}[node distance=2cm, auto]
        \node (w) [node distance = 3cm, left of = u] {$\mathbb C^7$};
        \node (t) [below of = w] {$\mathbb C^4$};

        \draw [->] (w) to [bend left] node {$C$} (t);
        \draw[->] (t) to [bend left] node {$B$} (w);
        \path[->] (w) edge [loop, out = 45, in = 135, looseness = 4, swap] node[auto] {$A$} (w);
        \path[->] (t) edge [loop, out = 225, in = 315, looseness = 4, swap] node[auto] {$D$} (t);

        \node (u) {$\mathbb C^7$};
        \node (v) [below of = u] {$\mathbb C^4$};

        \draw [->] (u) to node {$f_{21}(X)$} (v);
        \path[->] (u) edge [loop, out = 45, in = 135, looseness = 4, swap] node[auto] {$f_1(X)$} (u);

        \draw[->] (w) to node {$f$} (u);
        \draw[->] (t) to node {$f$} (v);

    \end{tikzpicture}
    \end{center}

Note that the objects are constant, it is the morphisms that are changed by this map. For $f$ to be free, then, because we have $\Gamma Y = X\Gamma$ we need $\Gamma f(Y) = f(X) \Gamma$, which is to say the squares over $\Gamma$ in the following diagram all commute. 

   \begin{center}
 \begin{tikzpicture}[node distance=2cm, auto]
        \node (w) [node distance = 3cm, left of = u] {$\mathbb C^2$};
        \node (t) [below of = w] {$\mathbb C^{13}$};

        \draw [->] (w) to node {$f_{21}(Y)$} (t);
        \path[->] (w) edge [loop, out = 45, in = 135, looseness = 4, swap] node[auto] {$f_1(Y)$} (w);

        \node (u) {$\mathbb C^7$};
        \node (v) [below of = u] {$\mathbb C^4$};

        \draw [->] (u) to node {$f_{21}(X)$} (v);
        \path[->] (u) edge [loop, out = 45, in = 135, looseness = 4, swap] node[auto] {$f_1(X)$} (u);

        \draw[->] (w) to node {$\Gamma_u$} (u);
        \draw[->] (t) to node {$\Gamma_v$} (v);

    \end{tikzpicture}
    \end{center}

    The equations associated with this are \[\begin{matrix}
        \Gamma_uf_1(Y) = f_1(X)\Gamma_u &\in B(\mathbb C^2, \mathbb C^7), \\ 
        \Gamma_vf_{21}(Y) = f_{21}(X)\Gamma_u &\in B(\mathbb C^2, \mathbb C^4).
    \end{matrix}\]
    If $f_1 = x_{12}x_2x_{21}-x_1x_{12}x_{21}x_1+2x_1^2$ and $f_{21} = x_{21}x_1^2+2x_2x_{21}x_1+x_2^2x_{21}$, then by using the equations of $\Gamma Y = X\Gamma$ we can verify that the equations \[\begin{split}f_1(X)\Gamma_u = (BDC- ABCA + 2A^2)\Gamma_u &= \Gamma_u( B'D'C'- A'B'C'A' + 2A{'}^{2})  = \Gamma_uf_1(Y)\\
    f_{21}(X)\Gamma_u = (CA^2 + DCA + D^2C)\Gamma_u &= \Gamma_v(C'A{'}^2 + D'C'A' + D{'}^2C) = \Gamma_vf_{21}(Y)
    \end{split}\]
hold. Thus $f$ is a free map. 

Note that this example has not relied on the nc subcategories, direct sums, or additivity in general; the free function we showed at the end required some notion of addition on the morphisms, and thus a pre-additive category. Free maps are a natural categorical construction, and deserve broader analysis than only on additive categories. From here on, however, we will assume again that $C$ is additive.

Additionally, we see that free maps are coordinate free; though we provide a labelling, the natural categorial notion that ``isomorphism is indistinguishable from equality" shows that the calculus herein is entirely independent of how we choose to label the morphisms. 

\begin{theorem}
    Let $C$ be an additive category, and let $U \subset C^Q$, $V\subset C^R$ be nc subcategories. A functor $F: U \rightarrow V$ is a free map if and only if it respects direct sums and natural automorphisms. 
\end{theorem}
\begin{proof}
   $\impliedby$ In any additive category the direct sum is well-defined and the morphism $f: A_1 \oplus A_2 \rightarrow B_1\oplus B_2$ can be decomposed $f = \m{f_{11} & f_{12} \\ f_{21} & f_{22}}$ where $f_{ij}: A_j\rightarrow B_i$. 
   Then for $X, Y\in C^{Q}$ consider $X \oplus Y$. If $\Gamma$ is a natural transformation $Y\rightarrow X$, then $\m{id_X & \Gamma \\ & id_Y}$ is an automorphic natural transformation $X\oplus Y \rightarrow X\oplus Y$; indeed, its inverse is $\m{id_X & -\Gamma \\ & id_Y}$, where we define $-\Gamma$ by taking $-\gamma$ for every morphism $\gamma$ making up $\Gamma$. This is well defined because $Hom_C(A, B)$ is an abelian group in any additive category $C$. 
   Then let $f_{ij}: X_j \rightarrow X_i$, and $g_{ij}: Y_j \rightarrow Y_i$ be any morphisms in $X, Y$ respectively, with $X_i, X_j \in ob(X)$, and similarly for $Y$, so that $X_i \oplus Y_i \in ob(X\oplus Y)$.  
   Then consider $f_{ij}\oplus g_{ij} = \m{f_{ij} & \\ &g_{ij}}$. We verify that $\m{id_X & \Gamma \\ & id_Y}$ is a natural transformation by checking \[\M{id_{X_i} & \Gamma _i \\ & id_{Y_i}}\M{f_{ij} & \\ & g_{ij}}=\M{f_{ij} & \\ & g_{ij}}\M{id_{X_j} & \Gamma _j \\ & id_{Y_j}},\]
which is true so long as $\Gamma_i g_{ij} = f_{ij}\Gamma_j$. 

Then if we assume that $F$ respects direct sums and automorphic natural transformations, we apply it to this equation and and immediately obtain $\Gamma F(g_{ij}) = F(f_{ij})\Gamma_j$, so $F$ is a free map.

   $\implies$ Let $F$ be a free map. Since it respects natural transformations, it respects automorphic natural transformations. 
   But then by using projective morphisms, namely $\m{id_X \\ 0}$, $\m{0 & id_Y}$, we can easily show that direct sums are preserved as well. 
\end{proof}

This sort of theorem is very valuable, in part because it gives us a more tractable characterization. We can check that something is free by checking only that it works with diagonal concatenation and conjugation; these are often reasonably easy things to verify. 

Note that this result holds in any additive category; we require many properties thereof, primarily direct sums, that hom-sets are abelian groups, and the ability to represent morphisms of direct sums as matrices. We do not require images or kernels to be unique, and thus do not require abelian categories; everything in this paper is applicable to general additive categories. 

Perhaps more important, however, is we can track how certain simplifying steps will be treated. Certain sorts of methods of linearization or block diagonalization --- common in realization theory, matrix analysis, and many other fields --- can be tracked through any free function, as evidenced by this result.   

\subsection{Types of Free Maps}

\begin{definition}
Let $f: C^Q \rightarrow C^R$ be a free map. By $f_j(X)$ we denote the image $f(X)(j)$ of the arc $j$ of $R$ over the functor $f(X)$. 

We say that a morphism in $X$ has \textbf{degree} $k$ if it is the image under $X$ of a path of length $k$. 
    \begin{itemize}
\item We say that $f$ is a \textbf{monomial} if each $f_j(X) = X(x^{k_j})$ for some morphism $x^{k_j}$ in $X$. 
\begin{itemize}
    \item It is said to be of degree $k$ if the highest degree morphism is degree $k$.
    \item It is said to be strictly of degree $k$ if all morphisms have degree $k$. 
\end{itemize}
\item We say that $f$ is a \textbf{polynomial} if each $f_j(X) = \sum_{i=0}^nk_ix^{\ell_i}$, where $k_i$ is an element of the appropriate abelian group.
\begin{itemize}
    \item It is said to be of degree $k$ if $k$ is the highest degree of any monomial appearing in any $f_j$. 
\end{itemize}
\item We say that $f$ is an \textbf{inversion map} if $f_j(X)$ is either an arc of $X$ or the inverse of an arc of $X$. 
\item We say $f$ is \textbf{rational} if it is a finite composition of polynomials and rational functions.
\begin{itemize}
    \item These form rational expressions, rather than rational functions per se. 
\end{itemize}
    \end{itemize} 
\end{definition}

This is certainly not an exhaustive list.

A more elegant description of non-commutative rational functions would certainly be nice, but such a description remains elusive. See [Rat] for an extensive description and discussion in the classical free analysis case. 

Consider a free map $f: C^Q\rightarrow C^R$. In some cases, $f$ can be fully characterized by some functor $f_*: R\rightarrow \tilde Q$. This relationship is encoded by the below diagram commuting, that is $f_*X= f(X)$.  $\tilde Q$ is some sort of enrichment of $Q$; we will use the nature of that enrichment to characterize our functions. 

\begin{center} \begin{tikzpicture}[node distance=2cm, auto]
  \node (C) {$C$};
  \node(b) [below of = C]{};
  \node (Q) [ left of=b] {$\tilde Q$};
  \node (R) [right of=b] {$R$};
  \draw[->] (Q) to node {$X$} (C);
  \draw[->] (R) to node[swap] {$f(X)$} (C);
  \draw[->] (R) to node {$f_*$} (Q);
\end{tikzpicture}\end{center}

This diagram makes no reference to the fact that $f$ is free, but it should not be forgotten.

In order to appropriately discuss these enrichments, we define the two following operations. 
\begin{definition}
    For an abelian group $k$, a \textbf{$k-$linear enrichment} of a category is consists of forming a module on each homset, taking the elements of the hom-set as generators. 
    
    In particular, if $Q$ is a free category, the $k-$linear enrichment of $Q$ will have each $hom(u,v)$ be a $k-$module with the morphisms in $mor(u, v)$ as the generators. 

    If we have a $k-$linear enrichment of a category, add some additional morphisms, and then take the $k-$linear enrichment of the new category, the old relationships will still be preserved. The new elements can now be added to the old. 

This is a specific instance of enrichment over an abelian category. 

A \textbf{formal inversion} of a category consists of adding formal inverses for every morphism that is neither a zero morphism nor already an isomorphism. 
\end{definition}

\begin{proposition}\label{structure}
Monomials, polynomials, and rational functions can be fully characterized by a functor $f_* :R\rightarrow \tilde Q$, as discussed above. 

In particular, they can be categorized as follows. 
\begin{itemize}
    \item $f$ is a monomial if and only if $\tilde Q$ can be taken to be $Q$. 
    \item $f$ is a polynomial if and only if $\tilde Q$ can be taken to be $\overline Q_k$, the $k-$linear enrichment of $Q$.
    \item $f$ is rational if and only if $\tilde Q$ can be formed using a finite sequence of $k$-linear enrichments and formal inversions. 
\end{itemize}
\end{proposition}

This proposition removes any qualms about considering free nc rational functions in terms of $k-$linear sums of formal noncommuting letters $x$. Thus, after this point, we will identify $f_*$ with $f$. 

\begin{proof}
    Consider a free nc polynomial $f$, with $f_j(X) = \sum_i k_i X^i$ for each $X$, and a functor $X$ mapping $x_j \mapsto X(x_j)$ for each arc $x_j$ of $Q$. 
    
    Note we define $X^k := X(x^k) = X(x_{i_1}\circ \dots \circ x_{i_k})=X(x_{i_1})\circ\dots \circ X(x_{i_k})$ where $x^k := x_{i_1}\circ \dots \circ x_{i_k}$ is a path in $Q$. 
    
    Now consider $f_*$ given by $f_{*j} = \sum_k k_kx^k$, which is a functor from $R \rightarrow \overline Q_k$. Then because of the relationship referenced above, for any given $X$ we have $f_{*j}X = \sum_k k_kX^k = f_j(X)$. Thus the polynomial $f$ is entirely identifiable with a functor $f_*: R\rightarrow \overline Q_k$. 

    This same argument, simply removing the sum, works a fortiori on monomials. 

    Since rational functions are already defined in terms of polynomials and inversion maps, for whichever $X$s that have the appropriate inverses, rational expressions are just polynomials on adjusted domains. 
\end{proof}
Might all free maps be characterized in the same way? 

As a result of all of this, we are justified in using the following notation. If $f$ is a polynomial, then we consider $f: Q \rightarrow R$ to consist of, at each arc $x \in ArcsR$, some  \[f_x = \sum_{i}k_iw_i\] where $w_i \in mor(s(x), t(x))$, $k_i\in k$. We then evaluate it on some $X \in C^Q$, by replacing $w_i$ by $X(w_i)$. 

Similarly, we can view rational functions as $f_x = \sum_ik_ig_{i}$
where $g_i$ is a product of polynomials and the inverses of polynomials, such that $g_i$ ends up parallel to $x$. 

While it is easy to see that monomials and polynomials (for an appropriate $k$) are free maps, it is far from obvious that an inversion map is a free map. We will prove that it is well defined on certain domains. 

\subsection{Rational Expressions and Regularity Domains}

To keep our notation easily comprehensible, we will refer to split monomorphisms as having left inverses, and split epimorphims as having right inverses. 

\begin{proposition}
    Choose $a$ morphism $a$ of $Q$. Assume there exists some $\tilde R$ that admits an additional morphism $b: t(a) \rightarrow s(a)$. 
    
    If $x = X(a)$ has a left (right) inverse for every $X \in U\subset C^Q$ 
    then the map $f:U\rightarrow C^{\tilde R}$ with $f(X) = X_{a^{-1}}$ is a free map.
\end{proposition}
\begin{proof}
  All that is required is to show that it respects direct sums and natural automorphism. For every object and morphism in $X$, this is trivial; thus we need only concern ourselves with $x^{-1}$, the left inverse of $x$, and paths containing it. 

    Consider the left inverse of $x\oplus y = \m{x&\\&y}$, where $y=Y(a)$ for some $Y\in U$. We have $\m{x^{-1} &\\&y^{-1}}\m{x&\\&y} = \m{1_x&\\&1_y}$, where these are the identities on the appropriate source objects. Thus $\m{x^{-1}&\\&y^{-1}} = x^{-1}\oplus y^{-1}$ is the left inverse of $x\oplus y$. 

    Now let $S$ be some natural automorphism, so that $(S^{-1}XS)(a)=s^{-1}xt$. Then we have $t^{-1}x^{-1}ss^{-1}xt=1,$ telling us that $t^{-1}x^{-1}s$ is the left  inverse. But this is exactly the appropriate relation for it to have when conjugated by $S$. 
    
    From this point, it is a simple exercise to check that any composition containing $x^{-1}$ also respects direct sums and natural automorphism, applying what was shown above. 

    The same method demonstrates the result on right inverses. 
    
    Thus on the appropriate domains, the inversion map is a free map.
\end{proof}

With this result in mind we make a brief aside about regularity domains; that is, domains on which particular rational expressions are well defined. 

\begin{definition}[regularity domains]
The regularity domain of a left (right) inversion map on the morphism $a$ is the induced subcategory $\mathcal R\subset C^Q$ where $X(a)$ has a left (right) inverse as a morphism of $C$ for every $X\in \mathcal R$. 

Let $g$ be a left (right) inversion map on $a$.

By $\mathcal R_U(g)$ we denote the sub-category of $U$ that is contained in $\mathcal R(g)$. If $U$ is an nc-subcategory, then so is $\mathcal R(g)$; indeed, if $X(a), Y(a)$ are left invertible, then so is $X(a)\oplus Y(a)$ and $S^{-1}X(a)S$.

The regularity domain of a rational function $f$ is the maximal induced subcategory $\mathcal R \subset C^Q$ in which all of the required inverses exist.

Because the image of an nc subcategory under a free map is an nc subcategory, we can apply the previous fact and see that the regularity domain of a rational function is an nc-subcategory. 
\end{definition}

With these several tools, there is a path to develop a proper theory of rational functions. The theory of germs transfers, allowing us to form equivalence classes defined in terms of agreement on domains; we then say that these equivalence classes are rational functions, as is standard in algebraic geometry and related topics. The details of this discussion are not critical to this paper, and will be dealt with in the future. 

One main hurdle seems to be that the determinant does not respect additivity, which stalls much of the machinery used in algebraic geometry to discuss rational functions. 

\subsection{Example 1: Free Analysis}

\subsubsection{The Functor Category}
Choose $Q$ to be the free monoid $\langle x, y\rangle$. This is the free category with one object $v$ and two generating arcs $x, y$. Choose $C$ to be the category of finite dimensional vector spaces over the field $k$.

Then a functor $X\in C^Q$ is fully defined by a choice of $X(v)= k^n$ for some $n$ and choosing $X(x), X(y) \in M_n(k)$, the set of $n\times n$ matrices over $k$. Each remaining morphism is sent to the corresponding product of matrices, such as  $xyx\mapsto X(x)X(y)X(x)$. The object $X = (X_1, X_2) \in  M_n^2(k)$ fully describes the functor;  the space $k^n$ is inferred from the action of the $X_i$, and the remainder of the morphisms are gained by exhaustive composition of these two. As such, we refer to a general point as $X\in M^2 := \bigcup_{n\in\mathbb N}M_n^2$. This is the typical notational convention in the classical literature on free analysis. 

\subsubsection{Free Subcategories}
A natural automorphism $S$ here consists of a matrix $S \in GL_n(k)$. 

Then we can also choose $Y \in C^Q$ with $Y(v) = k^m$, and $Y(x),Y(y)\in M_m(k)$. If both $X, Y$ are in the nc-subcategory $U$ of $C^Q$, then so are the following (as a representative sample). 

\begin{itemize}
    \item $X\oplus Y$ 
    \begin{itemize}
        \item  Object: $(X\oplus Y)(v) = X(v)\oplus Y(v) \cong k^{nm}$.
        \item  Morphisms: $(X\oplus Y)(x) = X(x)\oplus Y(x) = \m{X(x)&\\&Y(x)}$ and $(X\oplus Y)(y)$.
    \end{itemize}
    \item $S\ii XS$
    \begin{itemize}
        \item  Object: $(S^{-1}XS)(v) \cong X(v)$
        \item Morphisms: $(S\ii XS)(x) = S\ii X(x)S$ and $S\ii X(y)S$.
    \end{itemize}
    \item $T\ii(X\oplus Y)T$ for $T\in GL_{nm}(k)$
    \begin{itemize}
        \item  Object: $k^{nm}$ 
        \item Morphisms: $T\ii (X\oplus Y)T(x)=T\ii(X(x)\oplus Y(x))T = T\ii\m{X(x)&\\&Y(x)}T$.
    \end{itemize}
\end{itemize}

\subsubsection{Monomials and Polynomials}
As we will show here, the classical free polynomials are free maps. 

Consider the functors $X= (X_1, X_2) \in M_n^2(k)$, and $Y = (Y_1, Y_2)\in M_m^2(k)$. Our natural transformations $\Gamma$ must live within $M_{nm}$. We can write the relation $X\Gamma = \Gamma Y$ as $(X_1, X_2)\Gamma = \Gamma (Y_1, Y_2)$, that is $X_1\Gamma = \Gamma Y_1$ and $X_2\Gamma = \Gamma Y_2$. 

As long as these two morphisms are related in this way, then all the other morphisms are similarly associated; indeed, if $w \in \langle x_1, x_2\rangle$ is given by $w = x_{i_1}\dots x_{i_\ell}$ ($i_1, \dots, i_n \in \{1, 2\}$, then the corresponding morphism in $X$ is $X_{i_1}\dots X_{i_\ell}$. Thus \[\begin{split}X_{i_1}\dots X_{i_\ell}\Gamma &= \\ &= X_{i_1}\dots X_{i_\ell-1}\Gamma Y_{i_\ell}\\&\hspace{2mm}\vdots\\&= \Gamma Y_{i_1}\dots Y_{i_\ell}.
\end{split}\]
Thus a free monomial map is some $f: U \in C^Q \rightarrow C^R$ where $R$ is the free category with one vertex and $\alpha$ arcs $x_1, \dots, x_\alpha$ and $f(X) = (f_1(X), \dots, f_\alpha(X))$, and each $f_\alpha(X)$ is some morphism in $X$. 

A free polynomial map, then, is $f: U\rightarrow C^R$ with $f_i(X) = \sum_{j=1}^{i_i}k_jp_j$ where $k_j\in k$ and $p_j\in \langle X\rangle$. But this is free, because the linear transformations commute with elements of the base field. 

These results can all be verified by considering them in terms of direct sums and similarity; indeed, $p(X)\oplus p(Y) = p(X\oplus Y)$, and $p(S\ii XS)=S\ii p(X)S$ whenever $p$ is a monomial, and the rest follows similarly. Thus these properties are inherited from the additive properties of the category $C$. This inheritance will be further studied in section \ref{inheritance}.

\subsubsection{Inversion Map}
Consider the monoid $Q = \langle x\rangle$ with one generator, which is isomorphic to the free category with one object and one arc, and also to $\mathbb N$. Then let $X \in C^Q$ be invertible; that is, the morphism $X(x)$ is invertible. 

Now consider an inversion map $inv: Q\rightarrow R$ that retains $x$, but formally appends its inverse. We now have two generators, $x$ and $x^{-1}$, with of course the requirement that $x^{-1}x=1$; we now have that $R$ is isomorphic to $\mathbb Z$. This is now a group representation, with all of the restrictions thereof, as discussed in example 3. 

\subsubsection{Analytic Functions}
If we choose $\|\cdot\|$ to mean the maximum norm of any morphism in $X$, then $\|X\| \leq 1$ allows $f(X) = (f_1(X), \dots, f_\alpha(X))$ to be a free nc map if each \[f_i(X)=\sum_{j=1}^{\infty}k_jp_j\]
where $p_j\in \langle X\rangle$ and $\sum_{j=1}^\infty{k_j} <\infty$. Again, this is because $\Gamma$ commutes with elements of the base field. 

Similarly, we can define the standard exponential function $exp: U \in M\rightarrow M$ by $exp(x) = \sum_{n=0}^{\infty}\frac{x^n}{n!}$. We see that it is free because if $X\Gamma = \Gamma Y$ then $X^n\Gamma = \Gamma Y^n$, and all rational numbers commute with $X$. 

\subsubsection{A Rational Expression}
An example of a rational expression is $f:U\subset M_2 \rightarrow M_3$  given by $f(x,y) = (x\ii y^2, 3(yx-xy), x(y-x)\ii y)$. This rational expression can be constructed as $f(x,y) = g(i(j(\ell(x,y))))$ where \[\begin{split}
    &\ell: U\in M^2 \rightarrow M^3 \text{ with } \ell(x,y) = (x,y,y-x),
    \\\text{the inversion map } &j:U'\subset M^3 \rightarrow M^4 \text{ with } j(x,y,z) = (x,y,z,z\ii), \\
    \text{the inversion map } &i:U''\subset M^4\rightarrow M^5 \text{ with } i(x,y,z,w)=(x,y,z,w,x\ii), \\ \text{ and finally }&g:M^5\rightarrow M^3 \text{ with } g(x,y,z,w,v) = (vy^2, 3(yx-xy), xzy)
\end{split}\]
The regularity domain of this expression is $(X,Y) \in U \subset M^2$ so that $X, Y-X$ are both invertible; because these are matrices, there is no subtlety to deal with regarding left or right inverses. 

\subsection{Example 2: $\mathcal S$ch}
\subsubsection{Functor Categories}
Let $Q = \mathcal{S}ch$, and choose $C$ to be the category of finite dimensional vector spaces over $\mathbb C$. Note that we can choose much more exotic contexts, so long as we make sure they are additive categories. 

It is well worth examining how this example changes when $C$ is instead the category of abelian groups, with group homomorphisms as the morphisms, or the category of Banach spaces with linear maps as morphisms, or indeed the category with matrix tuples $M^d$ as objects with classical free functions as the morphisms.

One functor $X\in C^Q$ is the following.

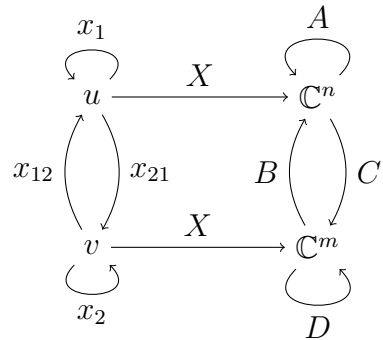
\begin{center}
    \begin{tikzpicture}[node distance=2cm, auto]
        \node (u) {$u$};
        \node (v) [below of = u] {$v$};

        \draw [->] (u) to [bend left] node {$x_{21}$} (v);
        \draw[->] (v) to [bend left] node {$x_{12}$} (u);
        \path[->] (u) edge [loop, out = 45, in = 135, looseness = 4, swap] node[auto] {$x_1$} (u);
        \path[->] (v) edge [loop, out = 225, in = 315, looseness = 4, swap] node[auto] {$x_2$} (v);

        \node (w) [node distance = 3cm, right of = u,] {$\mathbb C^n$};
        \node (t) [below of = w] {$\mathbb C^m$};

        \draw [->] (w) to [bend left] node {$C$} (t);
        \draw[->] (t) to [bend left] node {$B$} (w);
        \path[->] (w) edge [loop, out = 45, in = 135, looseness = 4, swap] node[auto] {$A$} (w);
        \path[->] (t) edge [loop, out = 225, in = 315, looseness = 4,swap ] node[auto] {$D$} (t);

        \draw[->] (u) to node {$X$} (w);
        \draw[->] (v) to node {$X$} (t);
        
    \end{tikzpicture}
\end{center}

All other morphisms in $\mathcal Sch$ are gained by composition of these operators; because they are freely associated, there is no restraint on our choice of operators.

As in the previous example, then, we can represent the entire functor as $X = (X_{11}, X_{22}, X_{12}, X_{22})$; the objects are implicit in the transformations. Here, however, there is some additional subtlety; if $n=m$ then $X(u)\cong X(v)$ and then $X_{11}X_{22}$ is a priori a fine product. However, since this product is generally not defined, it will not show up in a free map. Indeed, if it does then it must be that $X(u) = X(v)$ always, and a more appropriate category $Q$ should be used. 

\subsubsection{nc Subcategories}
A natural automorphism $S$ in this context is given by any pair of invertible transformations $S_u \in GL_n$ and $S_v \in GL_m$. 

Then if $Y \in C^Q$ is similar to $X$, with $\mathbb C^k$ and $\mathbb C^{\ell}$ rather than $\mathbb C^n$ and $\mathbb C^m$, and $X, Y$ both contained in an nc-subcategory $U$ of $C^Q$, the following are included in $U$. 

\[ \begin{matrix}X \oplus X & \hspace{1cm}& S^{-1}XS &\hspace{1cm}& X\oplus Y\\
    \begin{tikzpicture}[node distance=2cm, auto]
         \node (w)  {$\mathbb C^{2n}$};
        \node (t) [below of = w] {$\mathbb C^{2m}$};

        \draw [->] (w) to [bend left] node {$\m{C &\\&C}$} (t);
        \draw[->] (t) to [bend left] node {$\m{B &\\&B}$} (w);
        \path[->] (w) edge [loop, out = 45, in = 135, looseness = 4, swap] node[auto] {$\m{A&\\&A}$} (w);
        \path[->] (t) edge [loop, out = 225, in = 315, looseness = 4, swap] node[auto] {$\m{D&\\&D}$} (t);
    \end{tikzpicture} && 
    
    \begin{tikzpicture}[node distance=2cm, auto]
        \node (w)  {$\mathbb C^n$};
        \node (t) [below of = w] {$\mathbb C^m$};

        \draw [->] (w) to [bend left] node {$S_v\ii CS_u$} (t);
        \draw[->] (t) to [bend left] node {$S_u\ii BS_v$} (w);
        \path[->] (w) edge [loop, out = 45, in = 135, looseness = 4, swap] node[auto] {$S_u\ii AS_u$} (w);
        \path[->] (t) edge [loop, out = 225, in = 315, looseness = 4, swap] node[auto] {$S_v\ii DS_v$} (t);
    \end{tikzpicture}&&
    
     \begin{tikzpicture}[node distance=2cm, auto]
         \node (w)  {$\mathbb C^{n+k}$};
        \node (t) [below of = w] {$\mathbb C^{m+\ell}$};

        \draw [->] (w) to [bend left] node {$\m{C &\\&C'}$} (t);
        \draw[->] (t) to [bend left] node {$\m{B &\\&B'}$} (w);
        \path[->] (w) edge [loop, out = 45, in = 135, looseness = 4, swap] node[auto] {$\m{A&\\&A'}$} (w);
        \path[->] (t) edge [loop, out = 225, in = 315, looseness = 4, swap] node[auto] {$\m{D&\\&D'}$} (t);
    \end{tikzpicture} 
\end{matrix}\]
As well as \begin{itemize}
    \item $Y\oplus X$
    \item $Y\oplus S\ii XS$
    \item $\tilde S\ii (X\oplus Y) \tilde S.$
\end{itemize}

\subsubsection{Free Maps}
Here a natural transformation $\Gamma$ between $X,Y$ consists of a pair of operators $\Gamma_{u}: \mathbb C^k\rightarrow \mathbb C^n $ and $ \Gamma_v: \mathbb C^{\ell}\rightarrow \mathbb C^m*$ such that \[\begin{split}
    A\Gamma_u &= \Gamma_uA' \\B\Gamma_v&=\Gamma_uB'\\C\Gamma_u &= \Gamma_vC' \\D\Gamma_v &= \Gamma_vD'.
\end{split}\]
A monomial now corresponds to a path on $Q$, such as $T_pT_{pq}T_q = X(x_{11}x_{12}x_{22})$; here the natural transformation works as follows, varying from source to target of the operators. \[\begin{split}ABD\Gamma_v &= \\ &= AB\Gamma_vD' \\&= A\Gamma_uB'D' \\&= \Gamma_uA'B'D'\end{split}\]
Polynomials are a sum of such paths. An example of a polynomial $p:C^Q\rightarrow C^Q$ is \[\begin{split}p(x_{11},x_{12}, x_{22}, x_{21}) &=  (x_{11}x_{12}x_{21}+x_{11}, \hspace{3mm} x_{11}x_{12}x_{22}, \hspace{3mm}  x_{21}x_{12}+x_{22}^2, \hspace{3mm} x_{22}^4x_{21}+x_{22}^2x_{21}+x_{22}x_{21}) \end{split}\]
which is evaluated \[p(X) = (X_{1}X_{12}X_{21}+X_{1},  \hspace{3mm}  X_{1}X_{12}X_{2}, \hspace{3mm}  X_{21}X_{12}+X_{2}^2,  \hspace{3mm}  (X_{2}^4+X_{2}^2+X_{2})X_{21}),\] which is pictorially shown as 
\begin{center}
    \begin{tikzpicture}[node distance=2cm, auto]
        \node (u) {$\mathcal H_1$};
        \node (v) [below of = u] {$\mathcal H_2$};

        \draw [->] (u) to [bend left] node {$X_{21}$} (v);
        \draw[->] (v) to [bend left] node {$X_{12}$} (u);
        \path[->] (u) edge [loop above] node[auto] {$X_1$} ();
        \path[->] (v) edge [loop below] node[auto] {$X_2$} ();

        \node (w) [node distance = 5cm, right of = u] {$\mathcal H_1$};
        \node (t) [below of = w] {$\mathcal H_2$};

        \draw [->] (w) to [bend left] node {$(X_2^4+X_2^2+X_2)X_{21}$} (t);
        \draw[->] (t) to [bend left] node {$X_1X_{12}X_2$} (w);
        \path[->] (w) edge [loop above] node[auto] {$X_1X_{12}X_{21}+X_1$} ();
        \path[->] (t) edge [loop below] node[auto] {$X_{21}X_{12}+X_2^2$} ();

        \draw[->] (u) to node {$f$} (w);
        \draw[->] (v) to node {$f$} (t);
        
    \end{tikzpicture}
\end{center}

Note that we are now, as we often will in the future, using a notation that makes it more readily apparent which multiplications are allowed, and the relationship between the linear maps. Similarly, we have left the pure realm of finite complex vector spaces, and are writing as though we have more abstract Hilbert spaces. 

\subsubsection{The Schur Complement \& Principal Pivot Transform}
These are pivotal examples of rational maps. They are given by an inversion map together with a polynomial, with $Sch: C^Q \rightarrow C^{1\circ}$ given by $Sch(x) = x_{11}-x_{12}x_{22}\ii x_{21}$, where $1\circ$ is the free category generated by a loop on $u$ and no other arcs. It can also be considered $Sch: C^Q \rightarrow C^Q$ with the definition $Sch(x) = (x_{11}-x_{12}x_{22}\ii x_{21}, 0, 0, 0)$. 

\begin{center}
    \begin{tikzpicture}[node distance=2cm, auto]
        \node (u) {$\mathcal H_1$};
        \node (v) [below of = u] {$\mathcal H_2$};

        \draw [->] (u) to [bend left] node {$X_{21}$} (v);
        \draw[->] (v) to [bend left] node {$X_{12}$} (u);
        \path[->] (u) edge [loop, out = 45, in = 135, looseness = 4, swap] node[auto] {$X_1$} (u);
        \path[->] (v) edge [loop, out = 225, in = 315, looseness = 4, swap] node[auto] {$X_2$} (v);

        \node (w) [node distance = 5cm, right of = u] {$\mathcal H_1$};
        \node (t) [below of = w] {$\mathcal H_2$};

        \draw [->] (w) to [bend left] node {$0$} (t);
        \draw[->] (t) to [bend left] node {$0$} (w);
        \path[->] (w) edge [loop, out = 45, in = 135, looseness = 4, swap] node[auto] {$X_1-X_{12}X_2^{-1}X_{21}$} (w);
        \path[->] (t) edge [loop, out = 225, in = 315, looseness = 4, swap] node[auto] {$0$} (t);

        \draw[->] (u) to node {$Sch$} (w);
        \draw[->] (v) to node {$Sch$} (t);
        
    \end{tikzpicture}
\end{center}

The principal pivot transform, on the other hand, is the ``more complete" version of the schur complement, with $ppt: C^Q\rightarrow C^Q$ given by
\[ppt_B\M{A & B\\C&D} = \M{A\ii& -A\ii B\\ -CA\ii & D-CA\ii B} = \M{A-BD\ii C&-BD^{-1}\\-B\ii C&B\ii},\]
or pictorially, using the $X$ as above, 
\begin{center}
    \begin{tikzpicture}[node distance=2cm, auto]
        \node (u) {$\mathcal H_1$};
        \node (v) [below of = u] {$\mathcal H_2$};

        \draw [->] (u) to [bend left] node {$X_{21}$} (v);
        \draw[->] (v) to [bend left] node {$X_{12}$} (u);
        \path[->] (u) edge [loop, out = 45, in = 135, looseness = 4, swap] node[auto] {$X_1$} (u);
        \path[->] (v) edge [loop, out = 225, in = 315, looseness = 4, swap] node[auto] {$X_2$} (v);

        \node (w) [node distance = 5cm, right of = u] {$\mathcal H_1$};
        \node (t) [below of = w] {$\mathcal H_2$};

        \draw [->] (w) to [bend left] node {$-X^{-1}_2X_{21}$} (t);
        \draw[->] (t) to [bend left] node {$-X_{12}X_2^{-1}$} (w);
        \path[->] (w) edge [loop, out = 45, in = 135, looseness = 4, swap] node[auto] {$X_1-X_{12}X_2^{-1}X_{21}$} (w);
        \path[->] (t) edge [loop, out = 225, in = 315, looseness = 4, swap] node[auto] {$X_2^{-1}$} (t);

        \draw[->] (u) to node {$ppt_{1}$} (w);
        \draw[->] (v) to node {$ppt_1$} (t);
        
    \end{tikzpicture}
\end{center}

Since $X_2$  (or $B$) is a loop, as long as $X(v)$ is a Banach space then the regularity domain of the principal pivot transform is still simple, with no need to worry about left- or right-ness. 

\subsubsection{Exponential Maps \& the Campbell-Baker-Hausdorff Formula}
The exponential function can of course be applied here; because the exponential requires arbitrary powers, it can only be applied to cyclic paths. Thus $e^{X_{12}X_{21}}$ is well defined, but $e^{X_{12}}$ is not. 

Thus $f:C^Q \rightarrow C^Q$ with \[f\M{x_1&x_{12}\\x_{21}&x_{22}} = \M{e^{x_{12}x_{21}} & x_{12}+x_{12}e^{x_2} \\x_{21}&x_{21}e^{x_1}e^{x_{12}x_{21}}x_{12}}\]
is a free map. 

In general, \[e^{x_1}e^{x_{12}x_{21}} \neq e^{x_1 + x_{12}x_{21}}.\] Equality holds exactly when $x_1$ commutes with $x_{12}x_{21}$. The Campbell-Baker-Hausdorff formula provides the value of this product of exponentials. 

If $e^Z = e^Xe^Y$, then the first several terms of the formula are \[Z= X+Y + \frac{1}{2}[X, Y] + \frac{1}{12}[X,[X,Y]] - \frac{1}{12}[Y, [X,Y]] \dots\]
In the matrix case $[X, Y] = XY-YX$, and so this formula becomes \[\begin{split}
    Z &= X + Y + \frac{1}{2}(XY-YX) \\&+ \frac{1}{12}(XXY-XYX-XYX+YXX) \\&- \frac{1}{2}(YXY-YYX-XYY+YXY) \dots
\end{split}\]
and has (countably many) more linear combinations of such polynomials. This clearly has the form of a free map; when $X, Y$ are small enough that this converges, this is a free map. 

This belongs most naturally to the classical free analysis context, but has utility in our broader context.

\subsection{Example 3: Symmetric Group}
\subsubsection{Functor Categories}
Now let $Q = S_3$, the symmetric group on three elements. Let $C$ again be the category of Banach spaces. 

When finding a representation of $S_3$ we are more restricted, because we require all of our compositions to satisfy the relations in $S_3$. 

As a category, $S_3$ has one object, so we choose any Banach space $A$ to represent that object. 

Then each two cycle in $S_3$ must be sent by $X$ to some $x\in B(A)$ satisfying $x^2=id$. Thus we choose $(12)\mapsto x, (13)\mapsto y, (23)\mapsto z$ each its own inverse. But because $(123) = (23)(12) =(13)(23)=(12)(13)$, we require $X(123) = zx=yz=xy$, and similarly for $X(132)$. Thus there is a great deal of restriction placed. 

Indeed, if $A$ is chosen to be a finite dimensional vector space, then what we have done here is find a standard sort of matrix representation of $S_3$. 

\subsubsection{nc Subcategories}
A natural automorphism $S$ here is given by an invertible element $S \in B(A)$. 

If the functors $Y, X$ are both in the nc-subcategory $U$ of $C^Q$, then so are the following. 
\begin{itemize}
    \item $S\ii XS$ 
    \begin{itemize}
        \item Object: $A$
    \end{itemize}
    \begin{itemize}
        \item   Morphisms: $S\ii XS(12) = S\ii xS$, and similarly for the others. Thus $s\ii XS(123) = S\ii xyS = S\ii yzS= S\ii zx S$.
    \end{itemize}
    \item $X\oplus Y$
    \begin{itemize}
        \item Object: $X \oplus Y$ 
        \item Morphisms: $(12)\mapsto x \oplus x_Y = \m{x &\\& x_Y}$, and similarly for the others. 
        \item Thus $(123) \mapsto (z\oplus Y(23))(x\oplus X_Y) = (zx \oplus z_Yx_Y) = (yz \oplus y_Yz_Y)) = (xy \oplus x_Yy_Y = \m{xy &\\&x_Yy_Y}.$
    \end{itemize}
\end{itemize}
Because of the multiplicative properties of direct summation, we run into no issues here; all of those were sorted out in the original choice of the functors $X, Y$. 

\subsubsection{Free Maps}
Here, the free maps are much more restricted, at least when mapping group representations to group representations. 

One simple example is $f: C^{S_3} \rightarrow C^{S_2}$ given by $f(x,y,z) = x$. Of course, the constant function $f(x,y,z) = 1$ is also a free map. 

Similarly, there is a monomial maps $f: C^{S_3}\rightarrow C^{C_3}$ (where $C_3$ is the cyclic group of order 3) sending $g \mapsto xy$. Indeed, as an immediate consequence of Cayley's theorem [A], every finite group $G$ gives us an $n$ resulting in a free monomial map $f: C^{S_n}\rightarrow C^{G}$. 

We can also consider maps $f: C^{S_3} \rightarrow C^{A_3}$, where $A_3$ is the alternating group on three elements. The identity must be mapped to the identity, and the other two even elements can can be mapped either way. Indeed, any monomial is simply a homomorphism of group representations. 

Polynomials, however, must bridge contexts, as few polynomials will still be group representations. They can certainly be defined as free maps $f: C^{S_3} \rightarrow C^Q$ where $Q$ is a free category with one vertex. This is very similar to simply choosing an inclusion map $C^{S_3} \xhookrightarrow{} C^R$ where $R$ has three arcs, and then considering $f: C^R \rightarrow C^Q$. 

If $f$ is a polynomial, then $f(C^{S_3})$ at each arc of $Q$ has an element of the of the symmetric group algebra $k[S_3],$ for some appropriate commutative ring $k$ and some representation of $S_3$ over $C$. The symmetric group algebra $k[S_3]$ is extensively studied, such as in [S], [Sym]. The detail of these connections to the theory of representations is beyond the scope of this paper, but seems likely to be fruitful. 

\section{An Inverse Function Theorem}
\subsection{Differentiation}

\begin{definition}[convention]
    Choose a function $V: ob(Q) \rightarrow ob(C)$. We define $C^Q_V$ to be the subset of $C^Q$ where $u \mapsto V(u)$ for each $u\in ob(Q)$. 
    
    By the matrix \[A = \M{A_{11} & \dots & A_{1n}\\\vdots && \vdots\\A_{n1}&\dots &A_{nn}},\] with $A_{ij} \in ob(C^{Q}_V)$, we mean the functor that for each $x\in mor Q$ has

    \[A(x) = \M{A_{11}(x) & \dots & A_{1n}(x)\\\vdots && \vdots\\A_{n1}(x)&\dots &A_{nn}(x)}\]
    
   and has $A(u) = \bigoplus ^n V(u)$. This naturally maps $\bigoplus^nV(u) \rightarrow \bigoplus^n V(v)$ if $x: u\rightarrow v$. 
   \vspace{5mm}

    Similarly, let $X,Y \in ob(C^{Q})$ and fix a natural transformation $\Gamma: Y\rightarrow X$.
    
    Then $\m{1 & \Gamma \\& 1}:X\oplus Y\rightarrow X\oplus Y$ (typically notated $\m{1&\Gamma\\&1}: \m{X\\Y}\rightarrow \m{X\\Y}$ in a function space) is taken to mean the natural transformation formed of morphisms including the objects identity morphisms on the diagonal, and the natural transformation $\Gamma$ on the super-diagonal block. 
\end{definition}

With this convention in place, we can now discuss derivatives. 
\begin{definition}
    If $f: U \rightarrow W$ is a free map and $X\in U \subset C^Q_V$, $H\in C^Q_V$, we define the \textbf{directional derivative} $Df(X)[H]$ by the relation \[f\M{X & H \\ & X} = \M{f(X) & Df(X)[H]\\& f(X)}.\]
\end{definition}

It has been shown, such as in [KVV] and [IFT], that $Df(X)[H]$ is the classical derivative in the classical free analysis case. The same method extends to show that the it is still the classical derivative when $Q$ is a general free category, as will be shown here.

\begin{proposition}
  Let $C$ be a category of finite dimensional vector spaces. Let $C^Q, C^R$ be equipped with metrics induced by $C$, and $U \subset C^Q$ be open. 
  
  Let the free map $f: U\rightarrow V$ be continuous. Then $Df(X)[H]$ is the standard derivative, and thus $f$ is analytic. 
\end{proposition}

First we prove a simple lemma which will come in handy. 
\begin{lemma}
If the morphisms in the category $\mathcal X \in C^{Q}$ are of the form $x = \m{X & X\Gamma - \Gamma Y\\& Y}$ then \[f(x) = \M{f(X) & f(X)\Gamma - \Gamma f(Y)\\& f(Y)}\]
\end{lemma}
\begin{proof}
    Observe that $x$ is naturally automorphic to $\tilde x$ where $\tilde x = \m{X &\\&Y}$ by way of the natural transformation $s= \m{1 & \Gamma\\& 1}$. Evaluating $f(x) = s^{-1}f(\tilde x)s$ establishes the lemma.
\end{proof}

\begin{proof}[Proof of proposition]
    Choose $X \in U$. Then of course $X \oplus X\in U$. Because $U$ is open in a metric space, then we can choose sufficiently small $H$ so that $\m{X & H\\&X} \in U$. Similarly, we can choose a scalar $z$ small enough that \[Z(z) = \M{X +zH & H\\& X} \in U.\]

    If we choose a natural transformation $\Gamma = \frac{1}{z} I$, then we have that in fact \[Z(z) = \M{X+zH & (X +zH)\Gamma - \Gamma X \\& X}.\]
    Our lemma then tells us that \[f(Z(z)) = \M{f(X+zH) & \frac{f(X+zH)-f(X)}{z} \\& f(X)}.\]
    Since $f$ is continuous, in finite dimension this is exactly the standard derivative. 
\end{proof}

We demonstrate a chain rule relationship inherent in this definition; a Liebniz rule will be proven in \ref{product} after an appropriate product is defined. \label{liebniz}

Consider $(f\circ g)(X)$. Then $D(f\circ g)(X)[H]$ is given by evaluating \[(f\circ g)\m{X&H\\&X}=f\m{g(X)&Dg(X)[H]\\&g(X)}=\m{(f\circ g)(X) & Df(g(X))[Dg(X)[H]]\\& (f\circ g)(X)}\]
which give $D(f\circ g)(X)[H]=Df(g(X))[Dg(X)[H]]$, which looks much like the classical chain rule. 

It is no surprise that the free category case is very nice here as well; the following lemma shows that the derivative being zero is fully encoded in the generating arcs. This will greatly simplify our analysis going forward. 

\begin{proposition}
    If $R$ is a free category generated by arcs $x_1, \dots, x_n$ and $f: C^Q\rightarrow C^R$ is a free map, 
    
    then $Df(X)[H]=0$ if and only if $Df(X)[H](x_i)=0$ for all $i\in [n]$. 
\end{proposition}
\begin{proof}
    If $Df(X)[H](x_i)\neq 0$ for some $i$, then by definition $Df(X)[H] \neq 0$.  
    
    Conversely, if $Df(X)[H](x_i) = 0$ for all $x_i$, then $f\m{X&H\\&X}(x_i)=\m{f(X)(x_i)&\\&f(X)(x_i)}$; since every other morphism is a composition of these generators, the structure will remain block-diagonal and thus we see that $Df(X)[H]=0$.
\end{proof}

\subsection{Inverse Function Theorem}

Now, let us turn our attention to the inverse function theorem. We begin with a somewhat ambiguous structural lemma. In its ``general" form it requires a topological structure, specifically that of a metric space; this is easily come by when $C$ is a category of matrices, Hilbert spaces, Banach spaces, or a similar structure, though it is unnatural in some cases (such as the category of abelian groups). 

The first statement, however, is more often true, and is more crucial to the overall paper. 

\begin{lemma}
    Suppose $U \subset C^{Q}$ is an nc subcategory, and $f: U \rightarrow C^{R}$ is a free map. Also suppose that for $X, Y \in U$ we have \[f(X)\Gamma = \Gamma f(Y).\]
        $(1)$ If $f$ is injective on objects, then \[X \Gamma = \Gamma Y.\]
    
        $(2)$ Additionally, suppose that  $C^Q$ has a metric space structure and $U$ is open.
    If the preimage $f^{-1}\left(f(X) \oplus f(Y)\right)$ has compact closure in $U$ we have, 
    \[X\Gamma = \Gamma Y.\]
\end{lemma}
\begin{proof}
    First we prove $(1)$. Consider, as in our lemma, a category with morphisms $z = \m{X & X\Gamma-\Gamma Y\\ & Y}$.
    Then $f(z) = \m{f(X) & f(X)\Gamma - \Gamma f(Y) \\ & f(Y)} = \m{f(X) & \\ & f(Y)}$. Since $f$ is injective, this implies that $X\Gamma = \Gamma Y$.  

   $(2)$  If $U$ is open in a metric space over $k$, then for sufficiently small $t \in k$ we have $z(t) = \m{X & t(X\Gamma - \Gamma Y) \\& Y} \in U$. 
    But \[f(z(t)) = \m{f(X) & t(f(X)\Gamma - \Gamma f(Y)) \\ & f(Y)} = \m{f(X) & \\ & f(Y)} = f(z(0)).\]
Thus $\{z(t): t\in k\} \subset f\ii (z(0))$. Since the closure is compact, then, it must be a single point.
\end{proof}

Recall that we say $X=0$ only when it has zero in every component; it is no problem for $X$ to have exactly 1 non-zero entry. Thus we see that this is not a partial derivative in some sense, but rather a whole (directional) derivative. 

With these lemmata established, we can finally present the inverse function theorem.

\begin{theorem}[Inverse function theorem]
    Let $C$ be an additive category. Let $f:U\subset C^{Q} \rightarrow C^{R}$ be a free map. Then the following are equivalent. 
    \begin{itemize}
        \item For every $X \in C^{Q}$, the derivative $Df(X)[H] = 0$ only when $H = 0$.
        \item The map $f$ is injective.
        \item the inverse map $f^{-1}: f(U) \rightarrow U$ exists and is a free map. 
    \end{itemize} 
\end{theorem}
Note that saying $f$ is injective is the same as saying the functor $f$ is injective on objects. 
\begin{proof}
    Suppose that $Df(X)[H] = 0$ for some $H\neq 0$.  Then for $X \oplus X$, we have \[f\M{X & H\\ & X}= \M{f(X) & Df(X)[H]\\ & f(X)} = \M{f(X) & \\ & f(X)} = f\M{X &\\&X}.\] 
    Thus $f$ is not injective. \vspace{5mm}

    Now consider $Df(X)[H]$ to be zero only when $H = 0$. Now consider $X, Y$ so that $f(X) = f(Y)$. 
    If we consider $X\oplus Y\oplus X\oplus Y$, we can choose morphisms in it of the form \[\M{x &&& x-y\\&y&&\\&&x&\\&&&y} = \M{1 &&& -1\\&1&&\\&&1&\\&&&1} \M{x &&&\\&y&&\\&&x&\\&&&y}\M{1 &&& 1\\&1&&\\&&1&\\&&&1}\] where the $1$s are the identity morphism on the appropriate objects in $X$, $Y$. The morphisms $x, y$ must have the same target and source objects, because of the initial equation. 
    Evaluating $f$ on this, we see (due to the initial equation) that \[f\M{x &&& x-y\\&y&&\\&&x&\\&&&y} = \M{f(x) &&& \\&f(y)&&\\&&f(x)&\\&&&f(y)} = f\M{x &&&\\&y&&\\&&x&\\&&&y}\]
Because $f$ is injective, $x-y = 0$. 
Thus we have shown the equivalence of the first two statements. \vspace{5mm}

If the inverse map exists, it is clear that $f$ must be injective. 

Thus we assume that $f$ is injective. Since $f$ also fixes natural transformations, $f^{-1}$ is a functor; our lemma shows that in particular $f^{-1}$ is a free map. Indeed, if we have $W, V \in f(U)$ satisfying $W\Gamma = \Gamma V$, we have $f^{-1}(W)\Gamma = \Gamma f^{-1}(V)$. 
\end{proof}

This proof is basically the same as in [IFT], following all the same contours while ``under the hood" running on a more complicated framework. This demonstrates the principle that \textbf{natural proofs lift naturally}. With the appropriate sort of type checking, we get a much broader theorem from the same method of proof. 

An analogue to the weak Jacobian theorem as in [IFT] seems to follow in the case where $C$ is the category of finite dimensional vector spaces over $\mathbb C$, likely by blowing up the matrices with zeros in order to establish an analogue of a theorem in [KVV]. However, the golden goose here is a theorem that does not rely on representations over $\mathbb C^n$, so no more will be said on the matter here. 

\subsection{Proper Analytic Free Maps}
We now take a step back, and make use of the latter part of our earlier lemma, that assuming the existence of some metric on our additive category. Equipped with this metric space, we are able to extend the main theorem in [Prop]. 

The lifting of this theorem is a good example of how the extended free analytic theory can be restricted to prove powerful theorems in better understood spaces, as well as being another natural lifting of a proof. 

Thus, we will assume for simplicity that $C$ is some category of normed vector spaces. 

\begin{definition}
    We say that a function is \textbf{proper} when the pre-image of a compact set is compact. 

    We further say that a free map $f$ is \textbf{proper} if it is proper for each $C^Q_V$. 

    For $V, U \subset C^Q$ are free nc subcategories, with $C$ the category of finite dimensional vector spaces over the complex numbers, a function $f: U \rightarrow V$ is a \textbf{bianalytic free map} if $f^{-1}$ exists and both $f$ and $f^{-1}$ are proper analytic free maps. 
    
    Note that analytic here has no special sense beyond the normal, as we are in finite dimensional complex space. 
\end{definition}

\begin{theorem}
    Let $U \subset C^Q$ and $W \subset C^R$ be open nc subcategories. Let $f: U \rightarrow W$ be a free map. 
    \begin{enumerate}
        \item If $f$ is proper, it is injective and $f^{-1}: f(U) \rightarrow U$ is a free map.

        \item Suppose that for each $V: ob(Q)\rightarrow ob(C)$ and $Z\in C^Q_V$ the set $f^{-1}(Z) \subset C^R_V$ has compact closure in $U$.

        Then $f$ is injective and $f^{-1}:f(U)\rightarrow U$ is a free map.

        \item Suppose $Q = R$ and $C$ is a category of finite-dimensional vector spaces; suppose that $U, W$ both contain the all-zeroes functor.
        Then if $f$ is proper and continuous, $f$ is bianalytic.
\end{enumerate}
\end{theorem}
\begin{proof}
    (1) The first case is a special case of the second case, as singleton sets are compact. 

    (2) To prove the second statement, let us consider what happens when our natural transformation $\Gamma$ consists at each coordinate of $\gamma I$, for a sufficiently small scalar $\gamma$. Our earlier lemma then tells us that if $f(X) = f(Y)$, then also $X = Y$. Thus $f$ is injective on objects, and the inverse function theorem proves the rest. 

    (3) Because $f$ is continuous, we know that it is differentiable everywhere, and hence analytic on $C^Q_V$ for each $V: ob(Q) \rightarrow ob(Vec_{f.d.})$. But because it is also proper in each case, the theorem [F, 15.1.5] tells us that $f(U \cap C^Q_V) = W\cap C^Q_V$.  

    Then since $f$ is bijective and analytic, so is $f\ii$. Since $f$ is proper, $f\ii$ is free. Since $f\ii$ is analytic and bijective, it is proper. Thus $f$ is bianalytic. 
\end{proof}

This theorem and proof certainly begs to be made more general, though in all of its details it seems to require the great machinery of complex analysis that has been built up to this point. 

\subsection{Example: Free Analysis}
Consider the rational function $f(x,y) = (x\ii y^2, 3(yx-xy), y(y-x)\ii )$. If we evaluate it on $\tilde X = \left(\m{X&H\\&X},\m{Y&K\\&Y}\right)$, we get \[\begin{split}f(\tilde X) = \biggl(&\M{X\ii Y^2&-X\ii HX\ii Y^2+X\ii KY+X\ii YK\\&X\ii Y^2}, \\& \M{3(YX-XY)&3(KX+YH-HY-XK)\\&3(YX-XY)}, \\ &\M{Y(Y-X)\ii  &K(Y-X)\ii-Y(Y-X)\ii (K-H)(Y-X)\ii\\&Y(Y-X)\ii }\biggl)\end{split}\]
because $\m{X & H\\&X}\ii = \m{X\ii & -X\ii HX\ii\\&X\ii}$. Thus \[\begin{split}
    Df(X, Y)[H, K] = \big(-&X\ii HX\ii Y^2+X\ii KY+X\ii YK, \\&3(KX+YH-HY-XK),\\&K(Y-X)\ii-Y(Y-X)\ii (K-H)(Y-X)\ii\big).
\end{split}\]
If both $Y$ and $K$ are zero, the derivative is zero. Thus the inverse function theorem tells us that $f$ is not injective, nor an invertible function. 

\subsection{Example: $\mathcal S$ch}
\subsubsection{Schur Complement \& Principal Pivot Transform}
Now consider the schur complement as a free map $Sch: C^Q \rightarrow C^{1\circ}$ given by \[Sch(x) = x_{11}-x_{12}x_{22}\ii x_{21}.\] Let us evaluate on $X = \m{A&B\\C&D}$ for familiarity. Then $Sch(X)= A-BD\ii C$. Blowing this up with $\tilde A = \m{A & H_A\\&A}$, we have \[Sch(\tilde X) = \M{A-BD\ii C &H_A-H_BD\ii C+BD\ii H_DD\ii C-BD\ii H_C\\&A-BD\ii C}\]
resulting in \[DSch(X)[H] = H_A-H_BD\ii C+BD\ii H_DD\ii C-BD\ii H_C.\]
This is zero if $C=0, H_A = 0, H_C=0$. Thus $Sch$ is not invertible. 

On the other hand, when we consider the principal pivot transform we observe 
\[Dppt_D( X)[H] = \M{DSch(X)[H] & -D^{-1}H_B+D^{-1}H_DD^{-1}B \\ CD^{-1}H_DD^{-1}-H_CD^{-1} & -D^{-1}H_DD^{-1}}.\]

Because $D^{-1}$ is invertible, we see that immediately that $H_D$ must be zero if $Dppt_D(X)[H]$ is zero; thus $H_C, H_B$ must also be zero. If all three are zero, $H_A$ is forced to be zero as well. Thus we see that the principal pivot transform has an inverse which is also a free map. 

In fact, the inverse is $ppt_D$ itself. 

\subsection{Example: Symmetric Group}
Is it possible to have a functor $\m{x & h \\& x} \in C^{S_3}$, for a non-zero $h$? Certainly not as pure representations of $S_3$, or any symmetric group, as we must have $x^2 = 1$. Thus $\m{x & h\\& x}^2 = \m{x^2 & xh+hx \\& x^2} = \m{x^2 &\\&x^2}$. Thus either $x$ or $h$ must be zero, which breaks this. Thus in a pure theory of representations way, this inverse function theorem is of no use. 
Of course there still are invertible maps; any isomorphism of groups is one such. However, the inverse function theorem gives no information. 

We can, however, analyze such a map by way of an inclusion map $C^{S_3} \xhookrightarrow{} C^Q$, for some free category $Q$ with one object and three arcs. We then can view and analyze it as a classical free function; if this free function turns out to be invertible, then we know that on its original domain --- $U\subset C^{S_3}$ --- it is invertible, with free inverse. If, on the other hand, the modified function is not invertible, the current theory tells us nothing. Refinements would be quite interesting, but are not discussed in this paper. 

\section{Algebraic Structures of Free Functions}
\subsection{Vector Spaces of Free Functions}
\label{inheritance}

The assumption that every ``indexing category" $Q$ is a free category yields the clearest generalization of the classical theory of free analysis, as discussed in earlier examples. Nevertheless, if we only restrict the target category to be free, we have sufficient freedom from structure to do a great deal.

As a result of the lack of algebraic relations, we can define the following. 

\begin{definition}[addition and scalar multiplication on diagrams]
    Choose a free category $Q$ with the same objects as another category $R$. Also choose an additive category $C$ and a mapping $V: ob(Q) \rightarrow ob(C)$. Let us say $Arcs(Q) = \{x_1, \dots, x_n\}$.
    
    If $X, Y\in C^Q_V$, we define $X + Y \in C^Q_V$ by $(X+Y)(x_i)=X(x_i)+Y(x_i)$. This is well defined, because $C$ is an additive category. 

    We then define addition on free maps $f, g: C^R\rightarrow C^Q$ by $(f+g)(X)=f(X)+g(X)$; this addition commutes because $C$ is additive. \vspace{5mm}

We define scalar multiplication on $X\in C^Q$ over an appropriate scalar ring $k$, notated $a\otimes X$, by \[(a\otimes X)(x_i)=a\otimes X(x_i)\] where $a\in k$. Since any ring is a module over the integers, we can always consider $k = \mathbb Z$. 

This extends to the free maps in the same way, with $a\otimes f$ defined by \[(a\otimes f)(X) = a\otimes f(X).\]
    Often, such as when $k$ is the ring of integers or complex numbers,  we evaluate this multiplication as $(af(X))(x_i) = af(X)(x_i)$. Thus we have a $k$-module of free maps. 
\end{definition}

The following relationship remains true, as we would hope.

\begin{proposition}
    If $Q$ is a free category and $C$ is an additive category, every free nc polynomial $f: C^R \rightarrow C^Q$ is a finite sum of free nc monomials. 
\end{proposition}
\begin{proof}
Consider the free polynomial $f$, and let $f(X) = Y$. If we isolate the generators $x_1, \dots, x_n$, we can consider polynomials $f_1, \dots, f_n$ so that $f_i(X)(x_i)=Y(x_i)$ and $f_i(X)(x)=0$ if $x$ is any other morphism. 
    Now $f_i$ is still a polynomial, with $f_i(X)(x_i)=\sum_jk_{ij}y_{ij}$, $y_{ij}$ a morphism in $X$. We can then consider monomials $f_{ij}$ so that $f_{ij}(X)(x_i)=y_{ij}$, which is a monomial. 
    Then $f=\sum_{i,j}k_{ij}f_{ij}$ is a sum of monomials. 
\end{proof}
    
    With these two operations, we see that the free maps $\mathcal F_C(R, Q)$ always form a $\mathbb Z$-module, though depending on $C$ they are often modules over other rings. Indeed, the automatic abelian group structure of these modules leads us to wonder about the nature of a category of free maps. 

\subsection{The Polynomial Category}
    
\begin{definition}[the category of free polynomials on diagrams]
        Let $\mathfrak Q$ be the set of free categories $Q$ with the object set $ob(Q) = \mathcal V_0$, and let $C$ be an additive category. 

        The category $C^{\mathfrak Q}$ takes as objects the categories $C^Q$ with $Q\in \mathfrak Q$, and as morphisms the free polynomials $f:C^Q \rightarrow C^R$. 

        This is indeed a category because composition of free maps is well-behaved. It is a simple exercise to show that the composition of polynomials yields a polynomial. 
\end{definition}

Note that by using the addition defined above, inherited from $C$, the homsets in this category are all abelian groups. This suggests an additive structure to the polynomial category; we might then consider this structure to be analogous to the classical polynomial ring. 

\begin{proposition}
    Let $C$ be an additive category. Then the category $\mathcal F_C$ of free nc polynomials is an additive category. 
\end{proposition}
\begin{proof}
Every hom-set is a $\mathbb Z$-module and hence an abelian group, due to the addition inherited from $C$. To show that $\mathcal F_C$ is a pre-additive category, we must show that composition distributes over addition. 

As above, we can restrict ourselves to considering only the image of $Arcs(Q)$ in $C^Q$. Consider $f, g: C^Q \rightarrow C^R$, and $h: C^P\rightarrow C^Q$. Let $X \in P$. Then $h(X)\in C^Q$, and so by the definition of addition of free maps \[f(h(X))+g(h(X))=(f+g)(h(X)).\]
The left distribution is similar. Thus the abelian group action and the composition commute, and we see that this is a preadditive category. 

To establish that $\mathcal F_C$ is an additive category, we need to establish the existence of finitary biproducts. Due to preadditivity, products and co-products coincide, so we need to establish just one. 

Thus consider the product category $C^Q \times C^R = C^P$  where $P$ has generating arcs $A(P) = A(Q)\sqcup A(R)$. Indeed, the argument follows exactly the contours of that of $Set$, because the arcs are simply a finite set; when applied to the functor category $C^P$, rather than $P$, the projection monomials show that this case does not change.
\end{proof}

This leaves a few questions open. Might $\mathcal F_C$ be an abelian category, not just additive? If not, does $C$ being abelian guarantee that $\mathcal F_C$ is abelian as well? We shall not deal with these questions here, as abelian categories are not essential to this paper. 

The categories of rational maps, analytic maps, and general free maps are not discussed here. There are subtleties related to regularity domains, definition, and broad structure respectively that deserve more extensive treatment separately. 

The category of rational functions, for example, has issues with regularity domains that are reminiscent of the category of Banach spaces with unbounded linear operators. The rational functions seem to restrict us to some sub-category analogous to the general linear group, which of course is not closed under addition. 

\subsection{Rings of Free Maps}
\subsubsection{Multiplication on Diagrams}

Hitherto we have considered composition of free maps, which certainly allows a multiplicative structure on the set of maps $\mathcal F: C^Q\rightarrow C^Q$, giving us some sort of ring. Ye in the classical free analysis there is a simple element-wise multiplication structure, which we endeavor to emulate. 

\begin{definition}
    Let $P, Q, R$ all be categories with the same objects, and $C$ an additive category. 

    Then we define the product $\_\times_R \_: C^P\times C^Q \rightarrow C^R$ as a monomial map $m: C^P\times C^Q \rightarrow C^R$ of degree $2$ from the product functor category $C^P\times C^R$ (discussed in the above proof), where one of the arcs is from $P$ and the other from $Q$. 
    
    If the arcs from $C^P$ always appear to the left of the arcs from $C^Q$, then we say this multiplication is a left-multiplication. 

We then extend this definition to free maps in the following way. 

Consider $f: C^{\tilde P}\rightarrow C^P$ and $g: C^{\tilde Q}\rightarrow C^Q$ as well as a product $\_\times_R \_: C^P\times C^Q \rightarrow C^R$. Then we define $f \times_R g$ by 
\[(f\times_Rg)(X\times Y) = f(X)\times_Rg(Y)\]
for each $X \times Y \in C^{\tilde P}\times C^{\tilde Q}.$  
\end{definition}

As an example, let us consider $X \in C^Q$, where $C$ is a category of Hilbert spaces. Denote by $Q^*$ the free category that has $x^*: v\rightarrow u$ whenever $x: u\rightarrow v$ in $Q$. Further, we define the category $Q^*Q$ as having as arcs $x^*x$ for each $x\in Q$. 

Then if $Y \in C^{Q*}$ then we define $Y\times X$ has having $Y\times X(x^*x) = Y(x^*)X(x)$. This then is a product $\_\times\_: Q^*\times Q \rightarrow Q^*Q$. In particular, if we have $X^* \in C^{Q*}$ defined by $X^*(x^*)=X(x)^*$ (in the typical operator theoretic sense) then this product provides us hermitian squares $X^*X:= X^*\times X$.  

\begin{proposition}
    If $R$ is a free category, the product $\_\times_R\_: C^P\times C^Q\rightarrow C^R$ distributes over addition. 
\end{proposition}
\begin{proof}
    We want to show that $(f+g)\times_R h = f\times_R h+ g\times_R h$. 

    Consider an arc $x$ of $R$, for which $\times_R$ is described by  $x = yz$ where $y$ is an arc of $P$ and $z$ an arc of $Q$. Then the left here is \[(f+g)(X(y))\circ h(X(z))\]
    which boils down to distributivity in the additive category $C$, as all of these morphisms are ultimately morphisms in $C$. 
    All other situations are analogous. 
\end{proof}

And finally the promise of the Leibniz rule made in \ref{liebniz} can be kept, establishing simultaneously this multiplication and the derivative. \label{product} 

\begin{proposition}[Liebniz Rule]
    For a left multiplication $\times_R$, the liebniz rule $D(f\times_Rg)(X\times Y)[H] = Df(X)[H]\times_R g(Y) + f(X)\times_RDg(Y)[K]$ holds. 
\end{proposition}
\begin{proof}
    We have an equality $D(f\times_Rg)(X)[H] = Dm(X)[H]$, where $m$ is a degree 2 monomial from $C^P \times C^Q \rightarrow C^R$, where each arc has a component from $X \in C^P$ and one from $Y\in C^Q$. Note that because $X\times Y$ is an element of $C^P \times C^Q$, we know that each object is evaluated to the same thing in $X, Y$; that is, if $v$ is an object of $P$, then $X(v) = Y(v)$. 

    Consider the arc $x \in P$, and hence $m(Z)(x)$ for $Z = X\times Y$. This will look like $X(y)\circ Y(z)$ for $x=yz$, with $y\in Arcs(P)$ and $z \in Arcs(Q)$. Then if we evaluate $y$ on $\tilde X =\m{X & H_X\\&X}$ and $z$ on $\tilde Y = \m{Y & H_Y\\&Y}$ with $\tilde Z = \tilde X\times \tilde Y$, we have \[m(\tilde Z)(x) = \M{X(y)&H_X(y)\\&X(y)}\M{Y(z) & H_Y(z)\\&Y(z)} = \M{X(y)Y(z) & H_X(y)Y(z)+X_Z(y)H_Y(z) \\& X(y)Y(z)}.\]
This shows us the general picture; it is of course analogous for each arc, so we have \[Dm(\tilde Z)[H] = H_XY+XH_Y.\]

    Thus if we evaluate on $f: C^{\tilde P}\rightarrow C^P, g: C^{\tilde Q}\rightarrow C^Q$, we replace $X \in C^P$ with $f(X) \in C^P$ and $Y\in C^Q$ with $g(Y)\in C^Q$; we see that $\tilde X = \m{f(X) & Df(X)[H] \\ & f(X)}$ and $\tilde Y = \m{g(Y) & Dg(Y)[K] \\ & g(Y)}$ and we thus observe \[m(\tilde X\times \tilde Y) = \M{f(X)g(Y) & Df(X)[H]g(Y) + f(X)Dg(Y)[K]\\&f(X)g(Y)}\] which tells us that \[Dm(Z)[H] = D(f\times_Rg)(X\times Y)[H \times K] = Df(X)[H]g(Y) + f(X)Dg(Y)[K].\]
\end{proof}

For a more general multiplication a similar rule will hold, but at each arc of $R$ it will vary depending on which order the multiplication is there. The details can be inferred from the above proof. 

\subsubsection{Rings of Free Maps}
With left-multiplications thus defined, a fuller ring structure is finally available to us; we have had available a ring with multiplication as composition, but now we are able to construct a distinct multiplication more analogous to that of classical polynomials. 

Consider, then, the set of all free maps $f: C^Q \rightarrow C^P$. 
There is an addition structure, inherited from the additive category $C$. If we then select some left multiplication $\_\times_P \_: C^P \times C^P \rightarrow C^P$, we have a ring. 
Note that in general this ring is not unital, despite the standard example in free analysis being unital.

Because we have composition, the rings constructed in this way are in fact composition rings; they are even particularly nice composition rings, since composition distributes over addition. This draws a clear analogy to the univariate polynomial ring $k[x]$, the classic composition ring, though $k[x]$ is unital and an integral domain. 

This leaves a fruitful area of discovery open. What sorts of multiplications $\_\times_P\_$ will allow be unital? Indeed, what free category structures allow for unital categories; is it ever possible to construct a unital ring for some $Q$ where not all arcs are loops? What other nice properties can be extracted from certain products? What theorems (if any) on $k[x]$ can be replicated for these rings, whether on polynomials or otherwise? 

\subsection{Example 1: Free Analysis}
\subsubsection{Vector Space}
In the free analysis context we are dealing with all free categories that have exactly one object.

If we have two free functions $f, g: U \in M^3 \rightarrow M^2$, perhaps $f = (f_1, f_2)$ and $g = (g_1, g_2)$ they can be summed to get $f+g=(f_1+g_1, f_2+g_2)$. This is implicitly given by $(f+g)(X) = (f_1(X)+g_1(X), f_2(X)+g_2(X))$ for each $X \in U$, as these are all morphisms in the additive polynomial category. 

To obtain vector spaces, we simply need to involve scalar multiplication, with $\alpha(f+g) = (\alpha f_1+\alpha g_1, \alpha f_2 + \alpha g_2)$.

\subsubsection{Polynomial Category}
The objects of the polynomial category in the free analysis context are tuples of square matrices of all sizes, $M^d = \bigcup_{i\in \mathbb N}M_i^d(\mathbb C)$.

The morphisms of  the polynomial category consist of polynomial maps $f: M^d \rightarrow M^{\tilde d}$ given by $f= (f_1, \dots, f_{\tilde d})$ for all $d,\tilde d \in \mathbb N$, which is to say $\tilde d$-tuples of complex nc polynomials in $d$ variables. The addition, multiplication, and composition are standard entrywise nc polynomial addition, multiplication, and composition; that composition distributes over addition is easier to see here than in general. Call this category $\mathcal P_C$.

If $M^d = M^d(\mathbb C)$, then each homomorphism space $P(M^d, M^{\tilde d})$ is a complex vector space. 

The categorical product here is given by $M^d \times M^e = M^{d+e}$; in terms of the elements, we say that $X = (X_1, \dots, X_d) \in M_n^d$ and $Y = (Y_1, \dots, Y_e)\in M_n^e$ has $X\times Y = (X_1, \dots, X_d, Y_1, \dots, Y_e) \in M_n^{d+e}$. 

Thus $M^d \oplus M^e = M^{d+e}$. We can sum morphisms in the same way; if we have $f: M^d \rightarrow M^{\tilde d}$ and $g: M^e \rightarrow M^{\tilde e}$ we have $f\oplus g = (f, g): M^{d+e}\rightarrow M^{\tilde d +\tilde e}$. 

Let us then set $\mathcal P_C$ to be our base category, with $C$ the set of finite dimensional vector spaces, and study free maps $f: \mathcal P_C^Q \rightarrow \mathcal P_C^P$. 

Let us once again choose  $Q, P$ to be be $\langle x\rangle = \langle x_1, \dots, x_d\rangle$ and $\langle y\rangle=\langle y_1, \dots, y_{\tilde d}\rangle $ respectively. We construct this space analogous to the classical case by saying that $\mathcal M^d := \bigcup_{i}\bigoplus^d M^i$. The natural transformations will be polynomials $p: M^i \rightarrow M^j$, whereas the natural automorphisms here will be invertible free polynomials $p: M^i \rightarrow M^i$, for appropriate $i, j\in \mathbb N$. 

Such an $f: \mathcal M^d \rightarrow \mathcal M^{\tilde d}$ is given by $f(x) = (f_1(x), \dots, f_{\tilde d}(x))$ where each $f$ is a free map, such as a polynomial. This is completely unchanged from the previously discussed case, as a result of proposition \ref{structure}. 

The fact that each $x: M^e\rightarrow M^{e}$ is a free nc polynomial changes nothing here; this is not a surprise, as the morphisms that $\varphi \in \mathcal P_C$ deals with are themselves very often maps between function spaces. We are simply taking a further step. 

\subsubsection{Rings of Maps}
the typical multiplication is defined on $f, g: M^d\rightarrow M^c$ by $f\times_c g = (f_1g_1, \dots, f_cg_c)$, but this is by no means the only definition. 

This multiplication structure is unital; if we choose $e = (1, \dots, 1)$, then $f \times_c e = e \times_c f = f$. It is not in general an integral domain, 

Other multiplications are available; since the categorical product of $f:M^d\rightarrow M^c$ and $g:M^b \rightarrow M^a$ is given by $f\times g = (f_1, \dots, f_c, g_1, \dots, g_a) \in M^{d+a}$, a perfectly feasible left multiplication would be \[f\times_c g = (f_1g_1, \dots, f_cg_1).\]

There are many other possibilities, as all elements can be multiplied together because they are all in $Hom(X(v))$ for the vertex $v$ of $Q$ and any functor $X \in C^Q$.  

The standard free polynomial rings make use of the standard entrywise product, and it seems very unlikely that a different choice of product would be generally better. As a result, the multiplication is usually notated by $fg= f\times _c g $.

\subsection{Example 2: $\mathcal S$ch}
\subsubsection{Vector Spaces}
Each mapping $u \mapsto V(u) \in ob(C),v \mapsto V(v) \in ob(C)$,  defines an $R$-module, specifically \[Hom(V(u), V(v)) \oplus Hom (V(u), V(u)) \oplus Hom (V(v), V(v))\oplus Hom (V(v), V(u))\] Each of these is an abelian group, and they are all minimally $\mathbb Z-$modules. If $C$ is a category of vector spaces, say over $\mathbb C$, then the mapping $V$ defines a vector space. 

But these vector spaces allow us to define an $ R-$module or vector space over the free maps. Thus $\alpha f+ \beta g \in mor(C^Q\rightarrow C^R)$; we define it at $X$ by $\alpha f(X) + \beta g(X)$; at a given arc, say $p$, this is in particular $\alpha f_p(X) + \beta g_p(X)$. If $f_p(X) = PQP + P$ and $g_p(X) = PX + YP +P $, then \[(\alpha f + \beta g)_p(X) = \alpha PQP + \beta PX + \beta YP + (\alpha +\beta)P\]
which is certainly in the vector space $Hom(s(P), t(P))$. This can of course easily be described as a polynomial with abstract characters $x, y, p, q$, as long as those are known to be only populatable with products from a diagram of shape $Sch$. 

This is a context much like that of the quiver path-algebra, but we require source and target of all summands to match. It could be thought of as similar to a ``graded path algebra."

Thus in this general context we have $f_p = pqp+p$ and $g_p= px+yp+p$, and so $(\alpha f+ \beta g)_p = \alpha pqp +\beta px +\beta yp +(\alpha+\beta)p$; we see that it is really just the free polynomials with some additional requisite relations. 

\subsubsection{Polynomial Category}
See the polynomial category in the free analysis example. The generalization here is completely analogous to this paper's generalization of classical free analysis to the functor case. 

\subsubsection{Rings of Maps}
One possible multiplication $\_\times_{Sch} \_: C^{Sch}\times C^{Sch} \rightarrow C^{Sch}$ is defined by $f(X) \times_{Sch} g(X)$ mapping the arcs in the following way. \[\begin{split}
    x \mapsto f_x(X)\circ g_x(X)\\y\mapsto f_y(X) \circ g_y(X)\\p \mapsto f_p(X) \circ g_x(X)\\q\mapsto f_q(X) \circ g_y(X)
\end{split},\]
or pictorially 

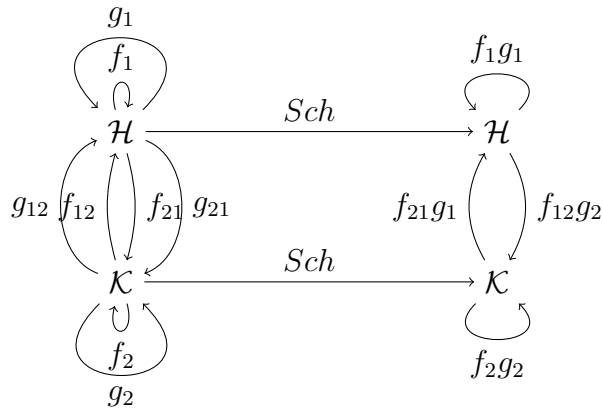
\begin{center}
    \begin{tikzpicture}[node distance=2cm, auto]
        \node (u) {$\mathcal H$};
        \node (v) [below of = u] {$\mathcal K$};

        \draw [->] (u) to [bend left=15] node {$f_{21}$} (v);
        \draw[->] (v) to [bend left=15] node {$f_{12}$} (u);
        \path[->] (u) edge [loop above=15] node[auto] {$f_1$} ();
        \path[->] (v) edge [loop below=15] node[auto] {$f_2$} ();
        
        \draw [->] (u) to [bend left =70] node {$g_{21}$} (v);
        \draw[->] (v) to [bend left = 70] node {$g_{12}$} (u);
        \path[->] (u) edge [loop, swap] node[auto] {$g_1$} ();
        \path[->] (v) edge [loop, out = 225, in = 315, looseness = 8, swap] node[auto] {$g_2$} (v);

        \node (w) [node distance = 5cm, right of = u] {$\mathcal H$};
        \node (t) [below of = w] {$\mathcal K$};

        \draw [->] (w) to [bend left] node {$f_{12}g_2$} (t);
        \draw[->] (t) to [bend left] node {$f_{21}g_1$} (w);
        \path[->] (w) edge [loop, out = 45, in = 135, looseness = 4, swap] node[auto] {$f_1g_1$} (w);
        \path[->] (t) edge [loop, out = 225, in = 315, looseness = 4, swap] node[auto] {$f_2g_2$} (t);

        \draw[->] (u) to node {$Sch$} (w);
        \draw[->] (v) to node {$Sch$} (t);
        
    \end{tikzpicture}
\end{center}
There has the nice advantage that if $g_x(X) = I_x$ and $g_y(X) = I_y$ then $f(X) \times_{Sch} g(X) = f(X)$. Thus if $g$ has $g_x = I$ and $g_y = I$ everywhere in the domain, $g$ will act like a right identity. It is not, however, a left identity. In fact, it seems quite unlikely that a left identity could exist for this multiplication; further, such a right identity $g$ is not at all unique. 

With such a multiplication, we have some sort of (non-unital) ring. We have a multiplication (with a right identity), we have addition, and we have composition. 

\subsubsection{An Integral Domain}
A somewhat more interesting ring is given for a family of functions $f: C^Q \rightarrow C^{1\circ}$. Here $1\circ$ is somewhat a misnomer; technically it must have $2$ objects, like $Q$. However, we simply give no arcs to or from one of the objects, and leave the other with exactly one loop.

Thus $f \times_{1\circ}g$ is just another map $C^Q \rightarrow C^{1\circ}$ given by a single multiplication of maps. Thus this ring is an integral domain. However, we have lost composition; $f$ does not take functors in $1\circ$.

\section{Final Questions}
We end by listing a few prominent questions. 

\begin{itemize}
    \item Are there any free maps that can't be characterized by a functor $f_*: R\rightarrow \tilde Q$?
    \item What do rational functions look like in this context?
    \begin{itemize}
        \item What is the appropriate topology?
    \end{itemize}
    \item What is the best version of Ax-Grothendieck and the Jacobian conjecture in this context?
\item Under what circumstances is it possible to define an addition for non-free categories?
\begin{itemize}
    \item Are there alternative structures that are more promising that grant analogous structure?
\end{itemize}
\item   What should a realization theory, as in[Rat], look like here?
\begin{itemize}
    \item Might the classical problem of isometries provide insight here?
\end{itemize}
\item Can we create a model-realization, and establish an Oka-Weil type theorem?
\item What is the appropriate sum-of-squares positivstellensatz as in [SoS] or [H17]?
\item What reasonable categories of free maps transcend the polynomial category? 
\begin{itemize}
    \item How closely related are they to the categories of Banach spaces with unbounded operators?
\end{itemize}
\item Under what circumstances do the polynomial rings from $C^Q \rightarrow C^P$ with product $\_\times_P\_$ behave like polynomial rings? 
\begin{itemize}
    \item In particular, which rings are unital, or integral domains? 
    \item What classes of products preserve which properties?

\end{itemize}
\item What can this tell us about the theory of representations? 
\end{itemize}
\newpage
\section*{References}
[KVV] Kalyuzhnyi-Verbovetski\i, D., and V. Vinnikov. ``Foundations of noncommutative function theory." \textit{preparation} 1, no. 2: 3.

[IFT] Pascoe, James E. ``The inverse function theorem and the Jacobian conjecture for free analysis." \textit{Mathematische Zeitschrift} 278 (2014): 987-994.

[PPT] Pascoe, J. E., and Ryan Tully-Doyle. ``Monotonicity of the principal pivot transform." \textit{Linear Algebra and its Applications} 643 (2022): 161-165.

[RSoS] Klep, Igor, James Eldred Pascoe, and Jurij Volčič. ``Regular and positive noncommutative rational functions." \textit{Journal of the London Mathematical Society} 95, no. 2 (2017): 613-632.

[Groth] Augat, Meric L. ``The free Grothendieck theorem." \textit{Proceedings of the London Mathematical Society} 118, no. 4 (2019): 787-825.

[SoS] Helton, J. William. ```` Positive" noncommutative polynomials are sums of squares." \textit{Annals of Mathematics} (2002): 675-694.

[SoHS] McCullough, Scott, and Mihai Putinar. ``Noncommutative sums of squares." \textit{Pacific Journal of Mathematics} 218, no. 1 (2005): 167-171.

[Inv] Pascoe, James Eldred. ``Invariant structure preserving functions and an Oka-Weil Kaplansky density type theorem." (2023): 465-482.

[H17] Volčič, Jurij. ``Hilbert’s 17th problem in free skew fields." In \textit{Forum of Mathematics, Sigma}, vol. 9, p. e61. Cambridge University Press, 2020.

[Prop] Helton, J. William, Igor Klep, and Scott McCullough. ``Proper analytic free maps." \textit{Journal of Functional Analysis} 260, no. 5 (2011): 1476-1490.

[Loci] Klep, Igor, and Jurij Volcic. ``Free loci of matrix pencils and domains of noncommutative rational functions." \textit{Comment. Math. Helv} 92, no. 1 (2017): 105-130.

[Sym] Grinberg, Darij, and Nadia Lafrenière. ``The one-sided cycle shuffles in the symmetric group algebra." \textit{arXiv preprint arXiv:2212.06274} (2022).

[Rat] Volčič, Jurij. ``Matrix coefficient realization theory of noncommutative rational functions." \textit{Journal of Algebra} 499 (2018): 397-437.

[A] Dummit, David Steven, and Richard M. Foote. \textit{Abstract algebra}. Vol. 3. Hoboken: Wiley, 2004.

[C] Riehl, Emily. \textit{Category theory in context}. Courier Dover Publications, 2017.

[F] Rudin, Walter. \textit{Function theory in the unit ball of Cn}. Springer Science \& Business Media, 2008.

[S] Grinberg, Darij. \textit{An introduction to the symmetric group algebra}.

\end{document}